\documentclass[11pt,reqno]{amsart}
\usepackage{graphicx} 
\usepackage[margin=1.5in]{geometry}
\usepackage{amsfonts, amssymb}
\usepackage{parskip}
\usepackage{amsthm}
\usepackage{amsmath}
\numberwithin{equation}{section}
\usepackage{fancyhdr}
\usepackage{algorithm}
\usepackage{algpseudocode}
\usepackage{mathtools}
\usepackage{subcaption}
\usepackage{hyperref}
\usepackage{parskip}

\usepackage{natbib}

\newtheorem{lemma}{Lemma}[section] 
\newtheorem{theorem}{Theorem} 

\newtheorem{proposition}{Proposition}

\theoremstyle{definition}

\newtheorem{assumption}{Assumption}[section] 
\theoremstyle{remark}
\newtheorem{remark}{Remark}[section]

\theoremstyle{definition}

\usepackage{blindtext}
\usepackage{hyperref}

\usepackage{tikz}
\usetikzlibrary{automata, arrows.meta, positioning}

\newcommand{\eps}{\varepsilon}

\newcommand{\R}{\mathbb{R}}

\newcommand{\cE}{\mathcal{E}}

\newcommand{\PP}{\mathbb{P}}
\newcommand{\QQ}{\mathbb{Q}}
\newcommand{\RR}{\mathbb{R}}

\newcommand{\ZZ}{\mathbb{Z}}

\newcommand{\bbr}{\RR}
\newcommand{\bbz}{\ZZ}
\newcommand{\bbq}{\QQ}
\newcommand{\one}{{\bf 1}}
\newcommand{\Var}{{\rm Var}}
\newcommand{\vep}{\varepsilon}

\hypersetup{
    colorlinks=true,
    linkcolor=blue,
    filecolor=magenta,      
    urlcolor=cyan,
    pdftitle={Overleaf Example},
    pdfpagemode=FullScreen,
    }

\usepackage{etoolbox, xcolor}
\newtoggle{answers}
\newcommand{\interior}[1]{%
  {\kern0pt#1}^{\mathrm{o}}%
}

\newcommand{\lims}{\lim\limits^{{-}\kern-3.8pt{-}}}
\newcommand{\limi}{\lim\limits_{{-}\kern-3.8pt{-}}}

\begin{document}

\title[Clustering of large deviations]{Clustering of large deviations in heavy-tailed moving averages:
  the catastrophe principle in the long-memory case}
\author{Jiaqi Wang, Gennady Samorodnitsky$^*$ }
\address{School of Operations Research and Information Engineering\\
Cornell University}
\email{jw2382@cornell.edu, gs18@cornell.edu}

\toggletrue{answers}

\thanks{ $^*$The corresponding author. Research is partially
  supported by NSF grant DMS-2310974 at Cornell University.}











\begin{abstract}
Clustering of large deviations events in a stationary stochastic
process depends critically on the 
interplay between the tail behavior of the marginal distribution and
the strength of temporal dependence. In the class of doubly infinite
moving average processes, when the memory is short,
previous work 
has established a sharp contrast in the clustering patterns between
light- and heavy-tailed  settings, governed by \textit{conspiracy} and
\textit{catastrophe} principles  respectively. It has also been
described how long memory interacts with the conspiracy principle to
affect clustering in the light-tailed case. This paper addresses the
interaction of long memory with the catastrophe principle in the
heavy-tailed case. It turns out that long memory generally allows for
a wider range of catastrophes to play a role and leads to longer and
qualitatively different clustering patterns. 
\end{abstract}

\maketitle

\section{Introduction}

We study clustering of large deviations events. These events are, by
definition, very rare, but may have a major effect. This effect is
multiplied when the large deviations event cluster, i.e. happen in 
relatively quick succession; one only needs to think of hurricanes
arriving in bands, earthquakes  in earthquake swarms, and market
drawdowns in flurries. Understanding how clustering of rare events
occurs is essential for assessing compound risk and designing
resilient systems.

Theory of large deviations is one of the most important and
established parts of probability theory and its applications, and it
involves studying not only how small the probabilities of large
deviations events are, but also how these events tend to occur. In
particular, it is well understood  that in stochastic
processes the mechanism of large 
deviations depends crucially on the ``tails''. When the tails are
``light'', large deviations events tend to occur due to
``conspiracy'', i.e. due to many small   coordinated distributional
shifts; see e.g. \cite{dembo:zeitouni:1998,
  deuschel:stroock:1989}. When the tails are ``heavy'',  by contrast,
large deviations events tend to be  driven by one or a few exceptional
values of the process —the so-called ``catastrophe principle''' see
e.g. \cite{denisov:dieker:shneer:2008, mikosch:winterberger:2013}. The 
study of clustering of large deviations is, however, relatively
new. The papers \cite{chakrabarty:samorodnitsky:2024,
  chakrabarty:samorodnitsky:2023} investigated how how large
deviations for partial sums of infinite moving average processes with
exponentially light tails cluster, thus shedding light on how the 
``conspiracy'' mechanism induces clustering, both in the short memory
and the long memory cases, with the latter situation demonstrating how the 
``conspiracy'' mechanism interacts with long memory. 

It turns out that, when the ``catastrophe principle'' applies, it
affects clustering of large deviation events in a completely different
manner. This was demonstrated in \cite{wang:samorodnitsky:2025} for
partial sums of infinite moving average processes with regularly
varying tails. In the short-memory regime considered there,
  a large deviation event is generated by a single large innovation,
  and this innovation is localized, on the linear scale, around the
  large deviation event. In fact, when linearly scaled, this 
location is asymptotically uniform over a finite range.  This causes 
 the resulting cluster of large deviation events to be limited in size
 and have relatively simple structure. The question of how the
 ''catastrophe principle'' interacts with long memory remained open. 

This paper aims to close this gap. We continue to consider partial
sums of moving average processes with regularly varying tails, but now
the process has long  memory. Large deviations are still 
explained by the ``catastrophe principle'', and so are still due to 
a single large innovation. However, because the dependence persists
over time,  the location  of the  dominant innovation is no longer
asymptotically uniform, and its effect of the  dominant innovation is
no longer localized. Instead, it reverberates through 
the process, creating clusters  whose span and shape are dictated by
the decay profile of the memory.


When a stationary process is described as an infinite moving average
process, it is broadly accepted, and is almost inevitable, to use the
rate of decay of the coefficients to describe the cases when the
memory is ``short'' and ``long''. This is what we do in this paper as
well. In other models, the length of memory may be better described
through different means. For example, when the stationary process is a
stable process or, more generally, an infinitely divisible process,
ergodic-theoretic-based distinction between short memory and long
memory may be appropriate; see e.g.
\cite{samorodnitsky:2004a,owada:samorodnitsky:2015a,
  chen:samorodnitsky:2022}, with a 
comprehensive description in \cite{samorodnitsky:2016}. We leave a
discussion of how ergodic-theoretic-based notions of long memory
affect clustering of large deviation to future work.

In this paper we consider $\RR^d$-valued infinite moving average processes
\begin{equation}\label{eq:def_moving_avg}
    X_k = \sum_{i=-\infty}^\infty A_i Z_{k-i}, \  k\in\ZZ,
\end{equation}
where $(Z_i)_{i\in\ZZ}$ are i.i.d.\ $\RR^d$ valued noise vectors, with generic element $Z$, and $(A_i)_{i\in\ZZ}$ are deterministic $d\times d$ coefficient matrices. 
We write $\|A\| = \|A\|_2$ for the $\ell_2/\ell_2$ operator norm of a matrix $A$, and all vector norms are Euclidean.


 As in \cite{wang:samorodnitsky:2025} we assume that 
the random vectors $(Z_i)$ are  regularly varying, with exponent
$\alpha>1$.  That is,
there exists a Radon measure $\nu$  on $\bar{\RR}^d \setminus \{0\}$
({\it   the tail measure}) such that the following vague convergence holds. 
\begin{equation}\label{eq:def_regular_varying_Rd}
    \frac{\PP(Z \in u \cdot)}{\PP(\|Z\| > u)} \xrightarrow{v}
    \nu(\cdot), \ u\to\infty, 
\end{equation}
The tail measure $\nu$ has the scaling property  $\nu(c \; \cdot) =
c^{-\alpha} \nu (\cdot)$ for any $c> 0$.  The assumption  $\alpha>1$
assures that the noise variables have a finite mean .
Throughout the paper we  assume  that $E[Z] = 0$. 

The matrices $(A_i)$ in \eqref{eq:def_moving_avg}, clearly, must
decay fast enough for the series $(X_k)_{k\geq 0}$ to converge. Sufficient conditions
are given in \cite{hult:samorodnitsky:2008}, and they are
\begin{align} \label{e:cond.conv}
&\sum_{i=-\infty}^\infty \|A_i\|^{\alpha-\vep}<\infty \ \ \text{for
                                    some} \ \ \vep>0 \ \ \text{if} \
                                    \ 1<\alpha\leq 2,\\
\notag &\sum_{i=-\infty}^\infty \|A_i\|^{2}<\infty \ \ \text{if} \
                                    \ \alpha>2.
\end{align}
It is common to view infinite moving average processes with
coefficients satisfying the assumption
\begin{equation} \label{e:summ}
\sum_{i=-\infty}^\infty \|A_i\|<\infty
\end{equation}
as having short memory; this is the assumption used in
\cite{wang:samorodnitsky:2025}. Clearly,   this is
a stronger condition than \eqref{e:cond.conv}.

One typically views an infinite moving average process
satisfying \eqref{e:cond.conv} but not the summability condition
\eqref{e:summ} as having long memory. It is in this situation that
\cite{chakrabarty:samorodnitsky:2023} analyzed clustering of large
deviations for processes with exponentially light tails. This is also
the situation we consider in the current paper.

We will impose an asymptotic assumption on the matrices
$(A_i)$. Specifically, we assume that
\begin{equation}\label{eq:def_coeff1}
   \lim_{m\to\infty} A_m/a_m = B_+, 
   \qquad 
   \lim_{m\to -\infty} A_m/a_{-m} = B_-
 \end{equation}
 for some $d\times d$ matrices $B_+$ and $B_-$, where
 $(a_m, \, m\geq 0) $ is a real-valued positive sequence that is regularly varying
 with exponent $\beta$. The exponent of regular variation is assumed
 to be in the range 
\begin{equation} \label{e:beta.tange}
    \beta \in \Big(-1,\; -\max\Big\{\tfrac{1}{\alpha}, \tfrac{1}{2}\Big\}\Big).
\end{equation} 

One immediately sees that, in this situation, the convergence condition
\eqref{e:cond.conv} holds,  but the summability condition
\eqref{e:summ} fails.

For the long memory infinite moving average process defined as
above, we study clustering of rare events related to  sums of
consecutive values of the process. We denote 
\begin{equation}\label{def:rareevent}
    E_j(n, \Gamma) = \left\{ S_j^n = \sum_{k=j}^{j+n-1} X_k \in \lambda_n \Gamma \right\}
  \end{equation}
for a measurable set $\Gamma$ bounded away from the origin (and
satisfying additional conditions to be stated in the
sequel). Furthermore, $(\lambda_n)$ is a sequence increasing to
infinity fast enough.  Specifically, we assume that
\begin{equation} \label{e:lambda.n}
    \liminf_{n\to\infty} \frac{\log \lambda_n}{\log n} >1 + \beta + \max\Big\{\tfrac{1}{2}, \tfrac{1}{\alpha}\Big\}.
  \end{equation}
   \begin{remark} \label{rk:lambda.seq}
 It is easy to see that, if a sequence $(\lambda_n)$ satisfies
 \eqref{e:lambda.n}, then the events $\bigl(E_j(n,\Gamma)\bigr)$ are
 rare events.  Indeed,  we can write
 \begin{equation} \label{e:sum.repr}
   S_0^n = \sum_{i=-\infty}^\infty
   \left(\sum_{j=-i}^{n-i}A_j\right)Z_i. 
 \end{equation}
If $\alpha>2$, then by the assumption \eqref{eq:def_coeff1} 
 \begin{align*}
E\|S_0^n\|^2 = \sum_{i=-\infty}^\infty E\left\|
   \left(\sum_{j=-i}^{n-i}A_j\right)Z_i\right\|^2 \sim c\,
   n^{2(\beta+1)+1} \ \text{as} \ \ n\to\infty
 \end{align*}
 for some $c>0$, so $\lambda_n^{-1}S_0^n\to 0$. Since the  set
 $\Gamma$ is bounded away from the origin, the events
 $\bigl(E_j(n,\Gamma)\bigr)$ are  rare. Similarly, if $1<\alpha\leq
 2$, then for  $p\in (1,\alpha)$ and sufficiently close to $\alpha$,
 by the by the Marcinkiewicz-Zygmund
inequality,
 $$
 E\|S_0^n\|^p \leq c_1 \sum_{i=-\infty}^\infty E\left\|
   \left(\sum_{j=-i}^{n-i}A_j\right)Z_i\right\|^p \sim c_2\,
   n^{p(\beta+1)+1} \ \text{as} \ \ n\to\infty
$$
for some $c_1,c_2>0$, so $\lambda_n^{-1}S_0^n\to 0$, and the events
 $\bigl(E_j(n,\Gamma)\bigr)$ are, once again, rare. 
\end{remark}

We finish this introduction with a brief description of the notation
used in the paper. For a set $\Psi\subset\RR^d$ and $\epsilon>0$ we
denote by  $\Psi^\epsilon = \{y: dist(\Psi, y) \leq \epsilon\}$ 
 the outer neighborhood of $\Psi$ and by $\Psi^{-\epsilon} = ((\Psi^c)^\epsilon)^c$
its inner neighborhood. Then 
    $\Psi^{-\epsilon} \subset \Psi \subset \Psi^{\epsilon}$, 
  $\Psi^{\epsilon} \downarrow \bar{\Psi}$ as $\eps \downarrow 0$
  and  $\Psi^{-\epsilon} \uparrow \interior{\Psi}$ as $\eps
  \downarrow 0$. As usual, $z_k$ is the $k$th coordinate of a vector
  $z$, and $A_{mk}$ is the entry of the matrix $A$ in the position
  $(m,k)$. Unless stated otherwise, the letter $C$ will denote a
  finite positive constant that may change from one appearance to to
  another. 

The rest of the paper is organized as follows. 
The main results are stated in Section~\ref{sec:main_result}. Section
\ref{sec:lemmas} contains a number of intermediate results used in the
paper. 
Sections~\ref{sec:prop_1_int_approx} and~\ref{sec:prop_2_prob_equi}
develop two key propositions that serve as the foundation for the
proofs of the main results, the proofs of which are then presented in Sections~\ref{sec:pt_process} and~\ref{sec:cluster_size}. 

\section{Main results}\label{sec:main_result}

To state our main results we introduce some notation. The following
matrix-valued function will play an important role in the sequel: 
\begin{equation}\label{eq:def_coeff_matrix}
    B_y = \big[(1-y)_+^{1+\beta} - (-y)_+^{1+\beta}\big] B_+ 
        - \big[(1-y)_-^{1+\beta} - (-y)_-^{1+\beta}\big] B_- , \ y\in\R.
\end{equation}

We will impose the following assumptions on the process $(X_n)$ and on
the ``failure set'' $\Gamma$ in \eqref{def:rareevent}
\begin{assumption}\label{ass:tail_measure}
$\Gamma \subset \RR^d \setminus
B_{\delta_0}(0)$ for some $\delta_0>0$,  $(pB_++(1-p)B_-)^{-1}\Gamma$ is a
$\nu$-continuity set for $p=0,1$ and almost every $0< p< 1$ (with respect to
the Lebesgue measure) and
$$
\int_\R \nu ( B_y^{-1}\Gamma)\, dy>0. 
$$
\end{assumption}

We will use the following notation for the partial sums of the
coefficients in \eqref{e:sum.repr} and the corresponding sums of the
numbers $(a_j)$: 
\begin{equation} \label{e:coeff.sums}
    H_i^{(n)} = \sum_{j=-i}^{n-i} A_j \in \RR^{d\times d},    \qquad 
    h_i^{(n)} = \sum_{j=-i}^{n-i} a_j, \ h^{(n)}=h_0^{(n)},
\end{equation}
with the convention $a_{-j}=a_j$ for all $j$.

Our first result should be compared to Theorem 2.1 in
\cite{wang:samorodnitsky:2025}. It formalizes the action of the
``catastrophe'' principle by stating that, asymptotically, a single
noise vector is responsible for the event $E_0(n,\Gamma)$, when it
occurs. It describes the ``size'' and the ``location'' of that noise
vector in the language of weak convergence of point processes. We give two related versions of such convergence. Let $M$
be the space of all Radon measures on 
$E=\bbr\times \bigl([-\infty,\infty]^d\setminus \{0\}\bigr)$, equipped with the topology of vague
convergence.  
We use
the common notation $\eps(x, y)$ to denote the Dirac measure (point
mass) at $(x, y)$.



Let $(\lambda_n)$ be a sequence satisfying
\eqref{e:lambda.n}, and consider the following two point processes: 
\begin{equation} \label{e:point.pr}
    N_n^1 = \sum_{i=-\infty}^\infty \epsilon\big(i/n, \, H_i^{(n)} Z_i/\lambda_n\big), \ \ 
     N_n^2 = \sum_{i=-\infty}^\infty \epsilon\big(i/n, \, h^{(n)} Z_i/\lambda_n\big).
\end{equation}
While the point process  $N_n^1$ is the natural candidates to describe the exceptionally large noise value, the point process $N_n^2$ offers extra flexibility for the subsequent analysis.

\begin{theorem}\label{thm:ptprocess_longmem}
Under Assumption \ref{ass:tail_measure},  the 
conditional laws of the point processes $N_n^i, i = 1, 2$ given $E_0(n,\Gamma)$, 
  converge weakly in  the vague topology, as $n\to\infty$,   to the
  laws  of a single atom point process $\epsilon_{(I,X^i)}, i = 1, 2$, with $I$
  real-valued with the density
  $$
  f_I(y) = \frac{\nu (B_y^{-1}\Gamma)}{\int_\R \nu (B_t^{-1}\Gamma)\, dt}, \ y\in\R
  $$
 with respect to the Lebesgue measure, and versions of the
 conditional laws of the $\R^d$-valued $X^i, i = 1, 2$  given  $Y=y$ are 
 $$
 P(X^1\in\cdot | I = y) = \frac{\nu (B_y^{-1} \Gamma \cap B_y^{-1} \; \cdot\;)}{\nu(B_y^{-1}
   \Gamma)}, \  \ P(X^2\in\cdot | I = y) = \frac{\nu (B_y^{-1} \Gamma \cap \; \cdot\;)}{\nu(B_y^{-1}  \Gamma)}. 
 $$
\end{theorem}
The reader may consult \cite{resnick:1987}
for more information on weak convergence in the vague topology.  

\begin{remark} \label{rk:jump.loc}
The fact that the limiting point process consists of a single point,
is a manifestation of the ``catastrophe principle'': a single
``large'' value of the noise is most likely responsible for the rare
event $E_0(n,\Gamma)$. This behaviour is also seen in the short memory
case. However, the location $I$ of the ``responsible'' noise value, while also of
order $n$, is no longer asymptotically uniformly distributed, and the
``size'' $X$ of the single large value is no longer asymptotically
independent of its location $I$, in
contrast to what is seen in the short memory case; see Theorem 2.1 in
\cite{wang:samorodnitsky:2025}. 

 It is instructive to look at the limiting distribution of the location
in a special case. Let  $d=1$; in this case the two point processes are, essentially, the same, so we look at $N_n^1$. We take $\Gamma=(1,\infty)$. In the
one-dimensional case the tail measure $\nu$ in
\eqref{eq:def_regular_varying_Rd} has the following density with
respect to the Lebesgue measure on $\R$:
\begin{equation} \label{e:nu.d1}
\frac{d\nu}{d\, {\rm Leb}}(x) = \left\{ \begin{array}{ll}
                                       p\alpha x^{-(1+\alpha)}
                                       &\text{if} \ x>0,\\
                                       q\alpha |x|^{-(1+\alpha)}
                                       &\text{if} \ x<0,
                                     \end{array}
\right.                                     
\end{equation}
for some $p=1-q\in [0,1]$. In this case the limiting location $I$ in
Theorem \ref{thm:ptprocess_longmem} has a density with
respect to the  Lebesgue measure on $\R$ given by
\[
    C\bigl(
    p\one(\phi_{B_+,B_-}(y)>0)+q\one(\phi_{B_+,B_-}(y)<0)\bigr)|\phi_{B_+,B_-}(y)|^\alpha,
    \ C>0,
  \]
with 
 \begin{equation} \label{e:phi.1}
    \phi_{B_+,B_-}(y) = B_+\big[(1-y)_+^{1+\beta} - (-y)_+^{1+\beta}\big] -
    B_-\big[(1-y)_-^{1+\beta} - (-y)_-^{1+\beta}\big], \ y\in\R,
\end{equation}
showing how non-uniform the limiting distribution is. 

Observe that the value  $\beta=-1$ forms  the boundary
  between summability \eqref{e:summ} and lack thereof. Therefore, 
it is natural to interpret this value as separating moving average
processes with short memory from those with long memory.  
As $\beta\downarrow -1$, 
\[
\phi_{B_+,B_-}(y)\to (B_++B_-)\one_{(0,1)}(y), \ y\not= 0,1.
\]
That is, as $\beta$ approaches $-1$, the density of the location $I$
approaches  the uniform density. This is consistent with the
fact established in \cite{wang:samorodnitsky:2025} 
that in the short memory case the corresponding density is, in fact,
uniform.

Interestingly,  In certain circumstances   $Y$ and $X$ may, in fact,
be independent. 
Suppose, for example, that $B_\pm = b_\pm B$ for some 
 real numbers $b_+>0$ and $ b_-\geq 0$ and matrix $B$ with
 $B^{-1}\Gamma$ being a $\nu$-continuity set of positive
 $\nu$-measure. 
Then  $ B_y =  \phi_{b_+,b_-}(y)B, \ y\in\R$. In this case the
function $\phi_{b_1,b_2}$ is nonnegative, so by the scaling
property of the measure $\nu$,
$$
\nu( B_y^{-1} \; \cdot\; )  = \nu(
B^{-1}\bigl(\; \cdot\; /\phi_{b_+,b_-}(y)\bigr) =\phi_{b_+,b_-}(y)^\alpha 
\nu(B^{-1} \; \cdot\; ).
$$
This means that 
the conditional law $\nu (B_y^{-1} \Gamma \cap \; \cdot\; )/\nu (B_y^{-1}
  \Gamma)$ of $X$ given $Y=y$ is independent of $y$ for a.e. $y$, and 
  in this case  $Y$ and $X$ are independent. The same is, of course,
  true if one interchanges the roles of $b_+$ and $b_-$, and of
  positivity and negativity. In particular, in the case $d=1$,
  independence occurs if $B_+$ and $B_-$ are not of opposite signs. In
  any dimension, independence holds for one-sided moving average
  processes.

\end{remark}

Our second result addresses the question of how many of the rare events
$E_j(n, \Gamma)$ happen given that the rare event $E_0(n, \Gamma)$ has
occurred. It is a specificity of the ``catastrophe principle'' that,
given that $E_0(n, \Gamma)$ occurs, an increasingly large number
$E_j(n, \Gamma)$  is likely to occur, both under the short memory
scenario and the long memory scenario. Therefore, in both cases it is
natural to measure the size of the cluster of large deviation events
  by studying the first time these events do not happen. For a
  measurable set $\Psi$ let 
\begin{align} \label{e:Js}
    &J_n^{\Psi,+} = \inf\{ j \geq 0 : \, E_j(n,\Psi)\ \text{does not
      occur}\},  \\
  \notag    &J_n^{\Psi,-} = \sup\{ j < 0 : \, E_j(n,\Psi)\ \text{does not
      occur}\}. 
\end{align}
The next theorem is stated in a greater generality; using it with
$\Psi=\Gamma$ yields a description of the size of the cluster of large
deviations.

\begin{theorem}\label{thm:cluster_size}  
Suppose that both $\Gamma$ and $\Psi$ satisfy
Assumption~\ref{ass:tail_measure}. 
Then for any $\theta_1,\theta_2>0$,
\begin{equation}\label{eq:clusterSize}
\begin{aligned}
    &\lim_{n\to\infty} 
    \PP\big(J_n^{\Psi,+} > n\theta_1,\, J_n^{\Psi,-} < -n\theta_2 \,\big|\, E_0(n,\Gamma)\big) \\
    &\quad= 
    \frac{\int_{-\infty}^\infty  \;
        \nu\!\Big(B_y^{-1} \Gamma \cap \bigcap^*_{-\theta_2\le t \le \theta_1} \!\!  B_{y-t}^{-1}\Psi\Big) dy}
        {\int_{-\infty}^\infty \nu(B_y^{-1} \Gamma) \,dy},
\end{aligned}
\end{equation}
where $\cap^*_{t\in A}$ stands for the intersection over all rational
$t\in A$. 
\end{theorem}

\begin{remark}
We interpret the inverses $B_\cdot^{-1}(\cdot)$ 
in \eqref{eq:clusterSize} as set inverses. Furthermore, the
matrix-valued function in \eqref{eq:def_coeff_matrix} satisfies
$$
\|B_y\|\leq C \Bigl(\big|(1-y)_+^{1+\beta} - (-y)_+^{1+\beta}\big|
+\big|(1-y)_-^{1+\beta} - (-y)_-^{1+\beta}\big|\Bigr)
$$
which along with the bounds on $\beta$ given in \eqref{e:beta.tange},
assures finiteness of the integral in the denominator. The situation
is similar with the integral in the numerator. 

Once again, it is instructive to look at the special case $d=1$. Let
us take now $\Gamma=\Psi = (1,\infty)$. The tail measure is still
given by \eqref{e:nu.d1}. Note that, if $B_+$ and $B_-$ are of the
same sign, then the function $\phi=\phi_{B_+,B_-}$ in \eqref{e:phi.1} is of a
constant sign, while if $B_+$ and $B_-$ are of different signs,  then
the function $\phi$ changes its sign once. Therefore, if
$\phi(y-\theta_1)$ and $\phi(y+\theta_2)$ are of the same sign, then
for all $-\theta_2\leq t\leq \theta_1$, $\phi(y-t)$ is of the same
sign as $\phi(y)$. Denote
$$
{\mathcal C}_{\theta_1,\theta_2}=\{ y\in\R:\, \phi(y-\theta_1) \
\text{and} \ \phi(y+\theta_2) \ \text{are of the same sign}\}
$$
(${\mathcal C}_{\theta_1,\theta_2}=\R$ if $B_+$ and $B_-$ are of the
same sign.) Then the limit in \eqref{eq:clusterSize} simplifies to 
\[ 
\frac{\int_{{\mathcal C}_{\theta_1,\theta_2}}\bigl(
  p\one(\phi_{B_+,B_-}(y)>0)+q\one(\phi_{B_+,B_-}(y)<0)\bigr) \bigl(
  |\phi_{B_+,B_-}(y-\theta_1)|^\alpha \wedge |\phi_{B_+,B_-}(y+\theta_2)|^\alpha\bigr)
   dy}
{\int_\R \bigl(
  p\one(\phi(y)>0)+q\one(\phi(y)<y)\bigr)|\phi(y)|^\alpha dy}. 
       \]


\end{remark}

\section{Integrals appearing in the
  limit}\label{sec:prop_1_int_approx}

The catastrophe principle essentially reduces the probability of the 
rare event \( \{S_0^{n}\in \lambda_n\Gamma\} \) to the sum of the
``large jump'' probabilities 
\[
\sum_{i\in\ZZ} P \big(H_i^{(n)}Z_i \in \lambda_n \Gamma\big),
\]
with the weights \(H_i^{(n)}\) given by \eqref{e:coeff.sums}. We will
spell out this reduction in the sequel. In this section we concentrate
on the sum above and certain more general sums, 
and show that, after a proper normalization, certain
integrals appear in the limit. These integrals also show up in the
main theorems of Section \ref{sec:main_result}.

We consider measurable functions $g:\, \R\times \R^d\to [0,\infty)$ of
the form 
\begin{equation} \label{e:simple.g}
g(y,x)= \sum_{l=1}^L \one_{[a_{l-1},a_l)}(y) \sum_{j=1}^{m_l} b_{lj}
\one_{A_{lj}}(x), \ y\in\R, \, x\in\R^d,
\end{equation} 
where $-\infty<a_0<a_1<\cdots <a_L<\infty, \, L=1,2,\ldots$,
$m_1,\ldots, m_L\in\bbz$, $(b_{lj})$ are nonnegative numbers and, for
each $l$, 
$(A_{lj}, \, j=1,\ldots, m_l)$ is a measurable partition of $\R^d$. The next proposition
is the main result of this section.

\begin{proposition} \label{prop:int_approx} 
 Suppose Assumption
  \ref{ass:tail_measure} holds and $(\lambda_n)$ satisfies
  \eqref{e:lambda.n}.

 (i)   If the set
  $(pB_++(1-p)B_-)^{-1} A_{lj}$ is a 
  $\nu$-continuity set for every $(l,j)$, $p=0,1$ and a.e. $0<p< 1$, then
   \begin{align}\label{eq:limit_sum_gen}
&\lim_{n\to\infty} \frac{\sum_{i = -\infty}^\infty E\bigl( e^{-g(i/n,
    H_i^{(n)} Z_i/\lambda_n)} \one_{\Gamma}(H_i^{(n)} Z_i/\lambda_n)\bigr) }
{n P(h^{(n)}\|Z\| > \lambda_n)}\\
\notag =&\frac{1}{\nu(\{z: \|z\|> 1\})}  \int_{-\infty}^\infty \left( 
\int_{B_y^{-1}\Gamma} e^{-g(y,B_yx)} \nu (dx)\right) \ dy.
   \end{align}

 (ii)     If the set $A_{lj}$ is a $\nu$-continuity set, then
 \begin{align}\label{eq:limit_sum_gen1}
&\lim_{n\to\infty} \frac{\sum_{i = -\infty}^\infty E\bigl( e^{-g(i/n,
    h^{(n)} Z_i/\lambda_n)} \one_{\Gamma}(H_i^{(n)} Z_i/\lambda_n)\bigr) }
{n P(h^{(n)}\|Z\| > \lambda_n)}\\
\notag =&\frac{1}{\nu(\{z: \|z\|> 1\})}  \int_{-\infty}^\infty \left( 
\int_{B_y^{-1}\Gamma} e^{-g(y,x)} \nu (dx)\right) \ dy.
 \end{align}

 In particular,
 \begin{equation}\label{eq:limit_sum_tot}
\lim_{n\to\infty} \frac{\sum_{i = -\infty}^\infty 
    P (H_i^{(n)} Z_i\in\lambda_n\Gamma)}
{n P(h^{(n)}\|Z\| > \lambda_n)}
=\frac{1}{\nu(\{z: \|z\|> 1\})}  \int_{-\infty}^\infty \nu(B_y^{-1}\Gamma)  dy.
\end{equation}
 \end{proposition}

Since the argument is similar for the two parts of the proposition,  we prove both parts at the same time. For this pupose we will use the notation 
\begin{equation}\label{eq:unify_coeff}
   G_i^{(n), 1} = H_i^{(n)}, \, G_i^{(n), 2} = h^{(n)}I, \, C_y^{1} = B_y, \, C_y^2 = I. 
\end{equation}
In the proof, we will omit the superscript $j$ and write
$G_i^{(n)}, C_y$ if a particular statement  holds for both cases, and
only add the superscript $j$ back when necessary. 

Denote 
\begin{equation}\label{def:a_y}
    a_+(y) =  (1 - y)_+^{1+\beta} - (-y)_+^{1+\beta}, \, a_-(y) = (1 -
    y)_-^{1+\beta} - (-y)_-^{1+\beta}, 
  \end{equation}
so that $B_y = a_+(y)B_+ - a_-(y) B_-$. It follows 
from Assumption \ref{ass:tail_measure}  that the sets
$C_y^{-1}\Gamma$, $C_y^{-1}A_{lj}$ 
are 
$\nu$-continuity sets for all $(l,j)$ and a.e. $y$.  We also note that
the matrix-valued  function
$y \mapsto C_y$ is continuous. 

We start with several lemmas. We replace for now the set $\Gamma$ by
any measurable set $\Psi 
\subset \RR^d \setminus B_{\delta_0}(0)$, and we  remove for now any 
continuity assumptions on the sets $(A_{lj})$. 
The following lemma allows us to group discrete terms in a sum into
groups of length $n\eps$ groups,  taking the first step of
approximating the sum by an integral.

\begin{lemma} \label{l:bounds}
For every $M> \max(|a_0|, |a_L|)$, 
$0<\eps<\min_{l=0,\ldots, L+1}(a_l-a_{l-1})$ and $0<\delta<\delta_0$, 
\begin{align} \label{e:block.upper}
  &\limsup_{n\to \infty} \, \bigl(n P(h^{(n)} \|Z\| > \lambda_
    n)\bigr)^{-1}\left[ \sum_{|i| 
          \leq Mn}E\bigl( e^{-g(i/n,    G_i^{(n)}Z_i/\lambda_n)}
  \one_{\Psi}(H_i^{(n)} Z_i/\lambda_n)\bigr) \right.\\
 \notag &\left. \hskip 0.05in - n\eps \sum_{l=0}^{L+1} \sum_{a_{l-1}/\eps\leq
                 k<a_l/\eps} \sum_{j=1}^{m_l} e^{-b_{lj}}P\bigl(
                 H_{[kn\eps]}^{(n)}Z_i/\lambda_n \in  \Psi^{\delta},  G_{[kn\eps]}^{(n)} Z_i/\lambda_n \in A_{lj}^\delta \bigr)\right] 
\leq f(\eps)  
  \end{align}
  and
\begin{align} \label{e:block.lower}
  &\liminf_{n\to \infty}  \, \bigl(n P(h^{(n)} \|Z\| > \lambda_
    n)\bigr)^{-1}\left[ \sum_{|i| 
          \leq Mn}E\bigl( e^{-g(i/n,    G_i^{(n)}Z_i/\lambda_n)}
  \one_{\Psi}(H_i^{(n)} Z_i/\lambda_n)\bigr) \right.\\
 \notag &\left. \hskip 0.05in - n\eps \sum_{l=1}^L \sum_{a_{l-1}/\eps\leq
                 k<a_l/\eps} \sum_{j=1}^{m_l} e^{-b_{lj}}P\bigl(
                 H_{[kn\eps]}^{(n)}Z_i/\lambda_n \in \Psi^{-\delta}, G_{[kn\eps]}^{(n)} Z_i /\lambda_n \in A_{lj}^{-\delta} \bigr)\right] 
\geq -f(\eps)  
  \end{align}
  for a nonnegative function $f$ with $f(\eps)\to 0$ as $\eps\to
  0$. Here $a_{-1}=-M, \, a_{L+1}=M, \, m_{-1}=m_{L+1}=1$ and
  $b_{-1,1}=b_{L+1,1}=0$. 
  \end{lemma}

\begin{proof}
 Clearly, for every $l=1,\ldots, L$,   any $i\in [a_{l-1} n,a_l n)$ 
and any $k$ 
  \begin{align*}
   &E\bigl( e^{-g(i/n,    G_i^{(n)}/\lambda_n)}
  \one_{\Psi}(H_i^{(n)} Z_i/\lambda_n)\bigr) 
 \leq  \sum_{j=1}^{m_l} e^{-b_{lj}} P\bigl(H_{[kn\eps]}^{(n)} Z_i/\lambda_n \in \Psi^{\delta}, G_{[kn\eps]}^{(n)} Z_i/\lambda_n \in A_{lj}^{\delta} \bigr) \\
 &  +  \sum_{j=1}^{m_l} e^{-b_{lj}} P(\|(H_i^{(n)} - H_{[kn\eps]}^{(n)} )Z_i\| >
    \lambda_n \delta)+  \sum_{j=1}^{m_l} e^{-b_{lj}} P(\|(G_i^{(n)} - G_{[kn\eps]}^{(n)} )Z_i\| >
    \lambda_n \delta). 
  \end{align*}
We complete the argument for $G_i^{(n)}= G_i^{(n),1}$,
the second case is even simpler.  We observe that for every $k$ and every
$i\in [ kn\eps,  (k+1) n \eps]$, 
$$\|H_i^{(n)} - H_{[kn\eps]}^{(n)}\|\leq \sum_{j\in [-(k+1)n\eps, -kn\eps]\cup
  [(1-(k+1)\eps)n, (1-k\eps)n]} \|A_j\|\coloneqq \tilde{A}_k^{(n)}, 
$$
say. By the Karamata theorem, for every fixed $k$,
\begin{equation}\label{eq:coeff_group_err}
\begin{aligned}
\limsup_{n\to\infty} \|\tilde{A}_k^{(n)}/h^{(n)}\| \leq&  \big|(|k| \eps)^{ 1 + \beta}  -
(|k+1|\eps)^{1 + \beta} \big| \\
+& \big||1 - k\eps|^{1 + \beta} - |1 - (k+1)
\eps|^{ 1 + \beta}\big|\max ( \|B _+\|,\|B_-\|), 
\end{aligned}
\end{equation}
 and the supremum of the right-hand side over $k\in\bbz$ (say,
 $f_1(\eps)$), 
 vanishes as $\eps\to 0$.
 Each $i$ with $(a_{l-1}+\eps)n\leq i<a_ln$ for some
$l=0,\ldots, L, L+1$ belongs to an interval $[ kn\eps,  (k+1) n \eps)$
for some $a_{l-1}/\eps\leq k<a_l/\eps$.  Therefore, 
\begin{align*}
&\limsup_{n\to \infty} \, \bigl(n P(h^{(n)} \|Z\| > \lambda_
    n)\bigr)^{-1}\left[ \sum_{|i| 
          \leq Mn}E\bigl( e^{-g(i/n,    G_i^{(n)}Z_i/\lambda_n)}
  \one_{\Psi}(H_i^{(n)} Z_i/\lambda_n)\bigr) \right.\\
 \notag &\left. \hskip 0.4in - n\eps \sum_{l=0}^{L+1} \sum_{a_{l-1}/\eps\leq
                 k<a_l/\eps} \sum_{j=1}^{m_l} e^{-b_{lj}}P\bigl(
                 H_{[kn\eps]}^{(n)}Z_i/\lambda_n \in \Psi^{\delta}, G_{[kn\eps]}^{(n)}Z_i/\lambda_n \in A_{lj}^{\delta}\bigr)\right] \\
\leq&  \limsup_{n\to\infty} \frac{n\eps\sum_{l=0}^{L} m_l \, P\bigl( \max_{|i|\leq Mn} \|
      H_i^{(n)}Z\|>\lambda_n\delta_0\bigr) + 4M n P(f_1(\eps)
           2h^{(n)}\|Z\| > \lambda_n
           \delta)}{n P(h^{(n)} \|Z\| > \lambda_n)}   \\
\leq& \lim_{n\to\infty} \frac{\eps\sum_{l=0}^{L} m_l \, P\bigl(
       4(\|B_+\| + \|B_-\|) h^{(n)}\| 
      Z\|>\lambda_n\delta_0\bigr) + 4M P(f_1(\eps)
           2h^{(n)}\|Z\| > \lambda_n
           \delta)}{ P(h^{(n)} \|Z\| > \lambda_
    n)}         \\
    =&\eps\sum_{l=0}^{L} m_l C \, \delta_0^{-\alpha}+  
\bigl(4f_1(\eps)/\delta\bigr)^{\alpha} =: f(\eps). 
\end{align*}
  This proves
\eqref{e:block.upper}, and \eqref{e:block.lower} can be proved in the
same way. 
\end{proof}

In the next lemma the limiting matrices $B_y$ make their first appearance. 
\begin{lemma}\label{lemma:asym_lim_meas}
Let $\Phi\subset \RR^d \setminus B_{\delta_0}(0)$. For every $k\in\bbz$, $\eps > 0$, $0< \delta<\delta_0$, 
 $$\limsup_{n\to \infty} \frac{  P({H_{[kn\eps]}^{(n)}} Z_k \in \lambda_n
   \Psi, G_{[kn\eps]}^{(n)} Z_k \in \lambda_n \Phi)}{P(h^{(n)} \|Z\| > \lambda_n)} \leq \frac{
   \nu(B_{k\eps}^{-1}\Psi^{\delta} \cap C_{k\eps}^{-1}\Phi^{\delta})}{\nu(\{z: 
   \|z\| \geq 1\})},
 $$
 $$
\liminf_{n\to \infty} \frac{  P({H_{[kn\eps]}^{(n)}} Z_k \in \lambda_n
   \Psi, G_{[kn\eps]}^{(n)} Z_k \in \lambda_n \Phi)}{P(h^{(n)} \|Z\| > \lambda_n)} \geq \frac{  \nu(B_{k\eps}^{-1}\Psi^{-\delta} \cap C_{k\eps}^{-1} \Phi^{-\delta})}{\nu(\{z: \|z\| \geq 1\})}.$$

 
\end{lemma}



 
\begin{proof}
For the first statement of the lemma we write 
         \begin{align*}
             & P(H_{[kn\eps]}^{(n)}Z_k \in \lambda_n \Psi, G_{[kn\eps]}^{(n)} Z_k \in \lambda_n \Phi) \leq P(B_{k\eps}
            h^{(n)} Z_k \in \lambda_n \Psi^{\delta/2}, C_{k\eps}h^{(n)}Z_k \in \lambda_n \Phi^{\delta/2})\\
            & + P(\|B_{k\eps}
              h^{(n)} - H_{[kn\eps]}^{(n)}\| \|Z_k\| > \lambda_n
              \delta/2) + P( \|C_{k\eps} h^{(n)} - G_{[kn\eps]}^{(n)}\| \|Z_k\| > \lambda_n \delta/2). 
        \end{align*}
The regular variation implies that 
$$\limsup_{n\to\infty}\frac{P(B_{k\eps}h^{(n)}Z_k\in \lambda_n \Psi^{\delta/2}, C_{k\eps}h^{(n)}Z_k \in \lambda_n \Phi ^{\delta/2})}{P(h^{(n)}
  \|Z\|> \lambda_n)} \leq \frac{\nu(B_{k\eps}^{-1}
  \Psi^{\delta} \cap C_{k\eps}^{-1}\Phi^{\delta})}{\nu(z: \|z\|\geq 1)}, $$
so we only need to show that the remaining terms are asymptotically of a
smaller order. Consider the case $k\leq 0$. By 
 \eqref{eq:def_coeff1} and the regular variation of
the sequence $(a_n)$,  we have $H_0^{(mn)}/h^{(n)}\to
m^{1+\beta}B_+$ for  $m\geq 0$. Writing 
$H_{[kn\eps]}^{(n)} =   H_0^{(-[kn\eps]+ n)} - H_0^{(-[kn\eps])}$, we see that
$H_{[kn\eps]}^{(n)} /h^{(n)}\to B_{k\eps}$. Therefore, by the regular
variation, 
\begin{align*}
&P(\|B_{k\eps}
              h^{(n)} - H_{[kn\eps]}^{(n)}\| \|Z_k\| > \lambda_n
              \delta/2) = P\bigl( h^{(n)}\| B_{k\eps}-H_{[kn\eps]}^{(n)} /h^{(n)}\|\|Z\|>\lambda_n
              \delta/2\bigr) \\
=&o\bigl( P(h^{(n)}
  \|Z\|> \lambda_n)\bigr),
\end{align*}         
and the second remaning term is of a smaller order as well.  Together
this proves the first statement of the lemma for $k\leq 0$, and the
case $k>0$ is similar. The second statement of the lemma can be proved
in the same manner.  
     
\end{proof}

The next lemma shows that the part of the sum in left-hand side of
\eqref{eq:limit_sum_tot} over $i$ too far from 0 is asymptotically negligible. 
\begin{lemma}\label{lemma:intapprox_negligible}
 For all $\delta > 0$, 
    $$\lim_{M\to \infty} \limsup_{n\to \infty}  \frac{\sum_{|i|> Mn}
      P(\|{H_i^{(n)}}Z_i\| > \lambda_n \delta)}{nP(h^{(n)} \|Z\| >
      \lambda_n)}= 0.$$ 
\end{lemma}
\begin{proof}
By \eqref{eq:def_coeff1}, there exists $C> 0$ such that
$\|A_i\|\leq Ca_i$ for all $i$, so that $\|H_i^{(n)}\|\leq Ch_i^{(n)}$
for all $i, n$. Since $Z$ is regularly varying, Potter's bounds say
that for any $s>0$, 
$$P(\|{H_{i}^{(n)}} Z_i\| > \lambda_n \delta)/P(h^{(n)}\|Z\|>
\lambda_n)  \leq C \bigl(h_{i}^{(n)}/h^{(n)}\bigr)^{\alpha - s}.$$
We can find a nonincreasing positive sequence $(b_n)$ such that $b_n/a_n\to
1$ as $n\to\infty$; see Theorem 1.5.3 in
\cite{bingham:goldie:teugels:1989}, which means that for $i<0$, 
$h_i^{(n)}\leq Cnb_{-i}\leq Cna_{-i}$, while for $i>n$, $h_i^{(n)}\leq
Cb_{n-i}\leq Ca_{n-i}$.  Further, by \eqref{e:beta.tange} 
 we can choose a small enough $s$ such that $(\alpha - s)\beta < -1$,
so that by Karamata's theorem, 
\begin{align*}
&\limsup_{n\to \infty} \frac{\sum_{|i|> Mn} P(\|{H_i^{(n)}} Z_i\| >
                 \lambda_n \delta)}{n P(h^{(n)}\|Z\|> \lambda_n)}\leq
                 C    \limsup_{n\to \infty} \frac{1}{n}\sum_{|i|> Mn}\bigl(h_{i}^{(n)}/h^{(n)}\bigr)^{\alpha - s}\\
  \leq &  C    \limsup_{n\to \infty} \frac{1}{n(h^{(n)})^{\alpha-s}}
         \left[ \sum_{i<-Mn}
         (Cna_{-i})^{\alpha-s}+
         \sum_{i>Mn} (Cna_{n-i})^{\alpha-s} \right] \\
\leq & C \limsup_{n\to \infty} \frac{n^{\alpha-s}}{n(h^{(n)})^{\alpha-s}}\sum_{i>(M-1)n} a_i^{\alpha-s}
\leq C \limsup_{n\to \infty}
       \frac{n^{\alpha-s}}{n(na_{n})^{\alpha-s}}(M-1)n (a_{\lceil (M-1)n\rceil})^{\alpha-s}\\
=& C(M-1)\lim_{n\to\infty} \bigl(a_{\lceil
   (M-1)n\rceil}/a_n\bigr)^{\alpha-s}=C(M-1)^{1+\beta(\alpha-s)} \to 0
   \ \ \text{as} \ M\to\infty.
\end{align*}

\end{proof}

We are now ready to prove Proposition \ref{prop:int_approx}. 

\begin{proof}[Proof of Proposition \ref{prop:int_approx}]
Let $M>0$.  By Lemmas \ref{l:bounds} and \ref{lemma:asym_lim_meas},
for any $\eps>0, \, \delta>0$, 
\begin{align*}
    & \limsup_{n\to\infty}  \frac{\sum_{|i|\leq Mn} E\bigl( e^{-g(i/n,
    G_i^{(n)} Z_i /\lambda_n)} \one_{\Gamma}(H_i^{(n)} Z_i/\lambda_n)\bigr) }
{n P(h^{(n)}\|Z\| > \lambda_n)}\\
\leq&\eps\sum_{l=0}^{L+1} \sum_{a_{l-1}/\eps\leq
                 k<a_l/\eps} \sum_{j=1}^{m_l} e^{-b_{lj}} \limsup_{n\to\infty}  \frac{ P\bigl(
                 H_{[kn\eps]}^{(n)}Z_i/\lambda_n \in 
                 \Gamma^\delta, G_{[kn\eps]}^{(n)} Z_i /\lambda_n \in \lambda_n A_{lj} \bigr)}
{P(h^{(n)}\|Z\| > \lambda_n)}\\
\leq& \eps\sum_{l=0}^{L+1} \sum_{a_{l-1}/\eps\leq
                 k<a_l/\eps} \sum_{j=1}^{m_l} e^{-b_{lj}}
\frac{\nu((B_{k\eps})^{-1} \Gamma^{2\delta} \cap (C_{k\eps})^{-1} A_{lj}^{2\delta})}{\nu(\{z: \|z\| \geq
      1\})} \\
=& 
\frac{1}{\nu(\{z: \|z\|\geq 1\})}
\int_{-M}^M \sum_{l=0}^{L+1}  \sum_{j=1}^{m_l} e^{-b_{lj}}\nu((B_{y}^{\eps})^{
      -1}\Gamma^{2\delta} \cap (C_y^{\eps})^{-1} A_{lj}^{2\delta}) \ dy,  
\end{align*}

where we denote $B_{y}^{\eps} = B_{k\eps},
C_y^{\eps} = C_{k\eps},$ for $ y\in [k\eps,
(k+1)\eps)$. Since $B_{y}^{\eps}\to B_y$, $C_y^{\eps}\to C_y$ as $\eps\to 0$ for every $y$,
and, for $\delta<\delta_0/2$,  the function $y\to \nu((B_{y}^{\eps})^{
      -1}\Gamma^{2\delta})$ is bounded uniformly in $y$ and $\eps$, we
    let $\eps\to 0$ and     conclude by Fatou's lemma that for any
    such $\delta$,

    \begin{align*}
&  \limsup_{n\to\infty}  \frac{\sum_{|i|\leq Mn} E\bigl( e^{-g(i/n,
    G_i^{(n)}Z_i/\lambda_n)} \one_{\Gamma}(H_i^{(n)} Z_i/\lambda_n)\bigr) }
                     {n P(h^{(n)}\|Z\| > \lambda_n)}\\
\leq& \frac{1}{\nu(\{z: \|z\|\geq 1\})}
\int_{-M}^M \sum_{l=0}^{L+1}  \sum_{j=1}^{m_l}
      e^{-b_{lj}}\limsup_{\eps\to 0}\nu((B_{y}^{\eps})^{
      -1}\Gamma^{2\delta} \cap (C_y^{\eps})^{-1} A_{lj}^{2\delta}) \ dy
      \\
   \leq&   \frac{1}{\nu(\{z: \|z\|\geq 1\})}
\int_{-M}^M \sum_{l=0}^{L+1}  \sum_{j=1}^{m_l}
      e^{-b_{lj}} \nu(B_{y}^{
      -1}\Gamma^{3\delta} \cap C_y^{-1} A_{lj}^{3\delta}) \ dy.
    \end{align*}

 Because of the $\nu$-continuity assumptions on the set $B_y^{-1} \Gamma$ and $C_y^{-1} A_{lj}$, we let $\delta\to 0$ to conclude that
    \begin{align*}
&  \limsup_{n\to\infty}  \frac{\sum_{|i|\leq Mn} E\bigl( e^{-g(i/n,
    G_i^{(n)} Z_i/\lambda_n)} \one_{\Gamma}(H_i^{(n)} Z_i/\lambda_n)\bigr) }
                     {n P(h^{(n)}\|Z\| > \lambda_n)}\\
   \leq&   \frac{1}{\nu(\{z: \|z\|\geq 1\})}
\int_{-M}^M \sum_{l=0}^{L+1}  \sum_{j=1}^{m_l}
      e^{-b_{lj}} \nu(B_{y}^{
      -1}\Gamma \cap C_y^{-1} A_{lj})) \ dy.
    \end{align*}
 By Lemma \ref{lemma:intapprox_negligible} we can let $M\to\infty$ and
 obtain 
    \begin{align*}
&  \limsup_{n\to\infty}  \frac{\sum_{i=-\infty}^\infty E\bigl( e^{-g(i/n,
    G_i^{(n)}Z_i/\lambda_n)} \one_{\Gamma}(H_i^{(n)} Z_i/\lambda_n)\bigr) }
                     {n P(h^{(n)}\|Z\| > \lambda_n)}\\
   \leq&   \frac{1}{\nu(\{z: \|z\|\geq 1\})}
\int_{-\infty}^\infty \sum_{l=0}^{L+1}  \sum_{j=1}^{m_l}
      e^{-b_{lj}}  \nu(B_{y}^{
      -1}\Gamma \cap C_y^{-1} A_{lj}) \ dy \\
=&      \frac{1}{\nu(\{z: \|z\|> 1\})}  \int_{-\infty}^\infty \left( 
\int_{B_y^{-1}\Gamma} e^{-g(y,C_yx)} \nu (dx)\right)\ dy.
    \end{align*}
 
A similar argument gives us a matching lower bound 
   \begin{align*}
&  \liminf_{n\to\infty}  \frac{\sum_{i=-\infty}^\infty E\bigl( e^{-g(i/n,
    G_i^{(n)}/\lambda_n)} \one_{\Gamma}(H_i^{(n)} Z_i/\lambda_n)\bigr) }
                     {n P(h^{(n)}\|Z\| > \lambda_n)}\\
\geq&      \frac{1}{\nu(\{z: \|z\|> 1\})}  \int_{-\infty}^\infty \left( 
\int_{B_y^{-1}\Gamma} e^{-g(y,C_yx)} \nu (dx)\right) \ dy, 
    \end{align*}
and the proof of the proposition is complete. 

\end{proof}
\begin{remark} \label{rk:prop1.gen}
For future use we note that we have actually proved that for any
measurable set $\Psi$ bounded away from the origin,
\begin{align} \label{e:less.prec.p1}
&\frac{\int_{-\infty}^\infty \nu\bigl((B_y^{-1}\Psi)^{\circ}\bigr)\
  dy}{\nu(\{z: \|z\| \geq 1\})} \leq  \liminf_{n\to \infty} \frac{
  \sum_{i= -\infty}^{\infty} {P({H_i^{(n)}} Z_i \in \lambda_n
    \Psi)}}{nP(h^{(n)}\|Z\|> \lambda_n)}\\
\notag \leq& \limsup_{n\to \infty} \frac{
  \sum_{i= -\infty}^{\infty} {P({H_i^{(n)}} Z_i \in \lambda_n
    \Psi)}}{nP(h^{(n)}\|Z\|> \lambda_n)}\leq
\frac{\int_{-\infty}^\infty \nu\bigl(\overline{B_y^{-1}\Psi}\bigr)\ 
  dy}{\nu(\{z: \|z\| \geq 1\})}. 
\end{align}

\end{remark}


\section{Catastrophe principle in action}\label{sec:prop_2_prob_equi}

The catastrophe principle describes how certain rare events are most
likely to happen. In this section, we make this precise by proving
asymptotic equivalence of certain probabilities. The proposition below
is the key result of this section; a similar statement holds in the
short memory case as it was shown in \cite{wang:samorodnitsky:2025}.
We use the notation $I_n^M := 
[-nM,nM]\cap\mathbb{Z}$, $M>0$, and we say that positive functions
$a(n,M)$ and $b(n,M)$ are asymptotically equivalent if
$$
\lim_{M\to\infty} \liminf_{n\to\infty} a(n,M)/b(n,M)=
\lim_{M\to\infty} \limsup_{n\to\infty} a(n,M)/b(n,M)=1.
$$

\begin{proposition}\label{prop:asym_prob_equi}
   Let \( M, \delta  > 0 \) and suppose $\Gamma$  
satisfies Assumption \ref{ass:tail_measure}. Then the following are
asymptotically equivalent: 


    \begin{enumerate} 
     \item $P(\text{there exists } i \in I_n^M \ \text{such that} \
       H_{i}^{(n)} Z_i \in \lambda_n \Gamma)$; 
        \item $\sum_{i = -\infty}^\infty P(H_{i}^{(n)} Z_i \in
          \lambda_n \Gamma)$; 
        \item $P(S_0^n \in \lambda_n \Gamma) $; 
        \item $P\bigl(S_0^n \in \lambda_n \Gamma \ \text{and there exists }
          i\in I_n^M \ \text{such that} \ H_i^{(n)} Z_i\in \lambda_n
          \Gamma\bigr)$;         
        \item 
$P\bigl(S_0^n \in \lambda_n\Gamma \ \text{and there exists } i\in
I_n^M \ \text{such that} \ H_i^{(n)} Z_i \in \lambda_n \Gamma \ \text{while
  for all} \   l\in I_n^M, \, l\neq i, \, \|H_l^{(n)} Z_l\| \leq \lambda_n
\delta\bigr)$. 
    \end{enumerate}
\end{proposition}

\begin{remark}
By the equivalence of $(2)$ and $(3)$ in Proposition
\ref{prop:asym_prob_equi} and Proposition \ref{prop:int_approx} we see
that 
    \begin{equation}\label{eq:asym_equivalence}
        {P(S_0^n \in \lambda_n \Gamma)} \sim \frac{n P(h^{(n)} \|Z\| >
          \lambda_n)}{\nu(\{z: \|z\|\geq 1\})}  \int_{-\infty}^\infty
        \nu (B_y^{-1}\Gamma) \, dy. 
    \end{equation}
\end{remark}

\medskip

Before proving the proposition, we first establish two lemmas.  Both
of them, effectively, rule out the possibility that more than one
exceptional noise variable contributes to the rare event. 

\begin{lemma}\label{l:elem.fact}
  For $M, \delta>0$ and measurable sets $\Phi, \Psi \subset \RR^d \setminus B_{\delta}(0)$,
\begin{equation} \label{e:two.big.sum}
\lim_{n\to\infty}\frac{\sum_{i \neq j\in I_n^M } P\bigl( H_i^{(n)} Z_i
  \in \lambda_n \Phi, \, H_j^{(n)} Z_j \in \lambda_n \Phi)}{\sum_{i\in I_n^M } P(H_i^{(n)} Z_i \in \lambda_n \Psi)}=0. 
\end{equation}
\end{lemma}

\begin{proof}
 Write
    \begin{align*}
       & \sum_{i \neq j\in I_n^M }P\bigl(H_i^{(n)} Z_i \in \lambda_n \Phi, \, H_j^{(n)}
         Z_j \in \lambda_n \Phi\bigr) \leq \left[\sum_{i  \in I_n^M }P\bigl(H_i^{(n)} Z_i
         \in \lambda_n \Phi\bigr)\right]^2, 
    \end{align*}
    By \eqref{e:less.prec.p1}, the last expression is bounded from
    above by the square of the function $f_n=nP(\|Z\|>
    h^{(n)}/\lambda_n)$ times a constant, and the
denominator in \eqref{e:two.big.sum} is bounded from below by $f_n$
times another constant. Since $f_n\to 0$ by \eqref{e:lambda.n}, the
claim of the lemma follows. 

\end{proof}

\begin{lemma}\label{lemma:prob_equi_sumDeviation} For any $\eps>0$, 
 \begin{equation}\label{eq:sum_deviation_alli}
     \lim_{n\to \infty} \frac{P(\|\sum_{l\neq i} H_l^{(n)} Z_l\| >
       \lambda_n \eps \ \text{\rm for all} \ i)}{\sum_{i  =
         -\infty}^\infty P(H_i^{(n)}Z_i\in\lambda_n \Gamma)} = 0. 
 \end{equation}
\end{lemma}
\begin{proof}
  For $\tau>0$ let $D_1 = \{\sup_i \|H_i^{(n)} Z_i\| >  \tau
  \lambda_n \}$, $D_2= D_1^c$. Let $(k_j)_{j=1}^\infty$ be an
  enumeration of $\bbz$. Denoting by $B$ the event in the numerator of
  \eqref{eq:sum_deviation_alli}, we have 
 \begin{align*}
   &P( B\cap D_1) = \sum_{j=1}^\infty P\bigl( B\cap\{ \|H_{k_j}^{(n)} Z_{k_j}\| >  \tau
  \lambda_n , \, \|H_{k_l}^{(n)} Z_{k_l}\| \leq  \tau
  \lambda_n \ \text{for} \ l=1,\ldots, j-1\}\bigr)\\
\leq& \sum_{j=1}^\infty  P(\|H_{k_j}^{(n)} Z_{k_j}\| > \lambda_n \tau)
      P(\|\sum_{l\neq {k_j}} H_l^{(n)} Z_l\| > \lambda_n \eps)\\
 \leq& \sum_{i  =
         -\infty}^\infty P(\|H_i^{(n)}Z_i\|>\lambda_n \tau) \sup_k
   P(\|\sum_{l\neq {k}} H_l^{(n)} Z_l\| > \lambda_n \eps). 
 \end{align*}
Since the supremum in the right-hand side vanishes in the limit (see 
Remark \ref{rk:lambda.seq}), and the sum is of the same order as
the sum in the denominator of \eqref{eq:sum_deviation_alli} by
Lemma \ref{lemma:asym_lim_meas}, we conclude that
$$
P(B\cap D_1) = o\bigl( \sum_{i  =
  -\infty}^\infty P(H_iZ_i\in\lambda_n \Gamma)\bigr).
$$

Next, on the set of $\{\|\sum_{l \neq i} H_l^{(n)} Z_l\|> \lambda_n \eps\}$,
there are $m,k\in \{1,\ldots, d\}$ such that
$  |\sum_{l \neq i}(H_l^{(n)})_{mk} (Z_l)_k| > \lambda_n\eps/d^{3/2}$.
Therefore, 
\begin{align*}
     P(B\cap D_2)
    \leq \sum_{m = 1}^d\sum_{k = 1}^d P\Bigl( \Bigl|\sum_{l \neq 0}(H_l^{(n)})_{mk} (Z_l)_k\Bigr| >
  \lambda_n\eps/d^{3/2}, \, |(H_l^{(n)})_{mk} (Z_l)_k| < \lambda_n \tau
  \Bigr). 
\end{align*}
We use Lemma \ref{lemma:WLLN_longmem} with
$\beta_{in}=|(H_i^{(n)})_{mk}|$.  For small enough $s>0$ we have
$$
\sum_{i =-\infty}^\infty \big|(H_i^{(n)})_{mk}\big|^{\alpha - s}
\leq Cn(h^{(n)})^{\alpha - s},
$$
and the conditions of the lemma are satisfied
thanks to the assumption \eqref{e:lambda.n}. If $1<\alpha\leq 2$, then by \eqref{eq:WLLN_bdd}
for small $s> 0$ we can further bound the right hand side above by  
\begin{align*}
    & \sum_{m = 1}^d\sum_{k = 1}^d  \left(\frac{(1+s)\sum_{i =
      -\infty}^\infty \big|(H_i^{(n)})_{mk}\big|^{\alpha - s}}{ (\eps/d^{3/2})
      \tau^{\alpha-s-1} \lambda_n^{\alpha - s}}\right)^{\delta/4\tau}
      \leq C \left(\frac{n(h^{(n)})^{\alpha - s}}{\lambda_n^{\alpha  -
      s}}\right)^{\delta/4\tau}. 
\end{align*}
By Proposition \ref{prop:int_approx} and \eqref{e:lambda.n} we
conclude that
$$
P(B\cap D_2) = o\bigl( \sum_{i  =
  -\infty}^\infty P(H_iZ_i\in\lambda_n \Gamma)\bigr).
$$
if $\tau>0$ is chosen to be small enough, proving the lemma in the case
$1<\alpha\leq 2$, and the case $\alpha>2$ is similar. 
\end{proof}

We are now ready to prove the main result of this section. 
\begin{proof}[Proof of Proposition \ref{prop:asym_prob_equi}] {\bf (1) $\sim$ (2)}: \ 
Only one asymptotic bounds has to be shown. By the inclusion exclusion formula, 
\begin{align*}
    &P\bigl(\text{there exists} \  i \in I_n^M \ \text{such that} \  {H_i^{(n)}} Z_i \in \lambda_n \Gamma\bigr) \geq \sum_{i\in I_n^M} P({H_i^{(n)}} Z_i \in \lambda_n \Gamma)) \\
    & - \sum_{l\neq i \in I_n^M} P\bigl({H_i^{(n)}} Z_i \in \lambda_n
      \Gamma,\,  {H_{l}^{(n)}} Z_l \in \lambda_n \Gamma\bigr)
    = (1 + o(1)) \sum_{i\in I_n^M} P\bigl({H_i^{(n)}} Z_i \in \lambda_n \Gamma\bigr)
\end{align*}
by Lemma \ref{l:elem.fact}. Letting $M\to\infty$ and using Lemma
\ref{lemma:intapprox_negligible} along with Proposition
\ref{prop:int_approx} establishes the required asymptotic bound. 

{\bf (2) $\sim$ (3)}: \ We proceed as  in the i.i.d. case considered
in \cite{hult:lindskog:mikosch:samorodnitsky:2005}.  By the
inclusion-exclusion formula, for any $\eps >0$, 
\begin{align*}
   & P(S_0^n\in \lambda_n \Gamma) \geq  P\Bigl(\bigcup_{i\in I_n^M }
     \Bigl\{H_i^{(n)} Z_i \in \lambda_n \Gamma^{-\eps}, \, \big\|\sum_{l\neq i} H_l^{(n)} Z_l\big\| < \lambda_n \eps\Bigr\}\Bigr) \\
     \geq& \sum_{i\in I_n^M } P(H_i^{(n)} Z_i \in \lambda_n
           \Gamma^{-\eps}) \, P\Bigl(\big\|\sum_{l\neq i} H_l^{(n)} Z_l\big\| < \lambda_n \eps\Bigr)\\
     -&  \sum_{i\neq l \in I_n^M} P\bigl(H_i^{(n)} Z_i \in \lambda_n
        \Gamma^{-\eps}\bigr) P\bigl( H_l^{(n)} Z_l \in \lambda_n \Gamma^{-\eps}\bigr).
\end{align*}
We know that $\inf_l P\Bigl(\big\|\sum_{l\neq i} H_l^{(n)} Z_l\big\| <
\lambda_n \eps\Bigr)\to 1$; see Remark \ref{rk:lambda.seq}. Along with
Lemma \ref{l:elem.fact} this gives us
$$
\liminf_{n\to \infty} \frac{P(S_0^n\in \lambda_n \Gamma)}{\sum_{i\in I_n^M} P(H_i^{(n)} Z_i\in \lambda_n \Gamma^{-\eps})}  \geq 1
$$
for all $M$ and all $\eps>0$. By Proposition
\ref{prop:int_approx} we may take $\eps\to 0$, and then by Lemma
\ref{lemma:intapprox_negligible} we may take $M\to\infty$ to obtain
the lower bound 
$$
\liminf_{n\to \infty} \frac{P(S_0^n\in \lambda_n \Gamma)}{\sum_{i =
    -\infty}^\infty P(H_i^{(n)} Z_i\in \lambda_n \Gamma)}  \geq 1.
$$
 
For an upper bound we write for $0<\eps<\delta_0$, 
\begin{equation*}\label{eq:prop_asymUB}
\begin{aligned}
    P(S_0^n \in \lambda_n \Gamma) \leq &\sum_{i = -\infty}^{\infty}
    P(H_i^{(n)}Z_i \in \lambda_n \Gamma^\eps) + P\bigl(S_0^n \in
    \lambda_n \Gamma, \, H_i^{(n)} Z_i \not \in \lambda_n
    \Gamma^{\eps} \  \text{for all} \  i\bigr) \\
    \leq & \sum_{i = -\infty}^{\infty} P(H_i^{(n)}Z_i \in
    \lambda_n\Gamma^\eps) +
    P\bigl(\|S_0^n - H_i^{(n)}Z_i\| > \lambda_n \eps \ \text{for all} \
    i\bigr). 
\end{aligned}
\end{equation*}
By Lemma \ref{lemma:prob_equi_sumDeviation} the second term is of a smaller order than
the first term.  Letting $n\to\infty$ and noticing that, by Proposition
\ref{prop:int_approx}  and Assumption \ref{ass:tail_measure} we may
then let $\eps\to 0$, establishing the required upper bound. 

 {\bf (3) $\sim$ (4)}: \ By Lemma \ref{lemma:prob_equi_sumDeviation}
 and the already established equivalence of {\bf (2)} and {\bf (3)}, 
for any $0<\eps<\delta_0$, 
\begin{align*}
   & \limsup_{n\to \infty} \frac{P\bigl(S_0^n\in
     \lambda_n\Gamma, \,  {H_i^{(n)}} Z_i \not \in   \lambda_n
     \Gamma^\eps \ \text{for all} \  i   \bigr)}{P(S_0^n\in \lambda_n\Gamma)}\\
    \leq&  \lim_{n\to \infty} \frac{P(\|\sum_{l\neq i} H_l^{(n)}Z_l\|>
     \lambda_n \eps \ \text{for all}\  i)}{\sum_{i =-\infty}^{\infty} P(H_i^{(n)} Z_i
     \in \lambda_n \Gamma)}  = 0. 
\end{align*}
Now our claims follows since by the already established equivalence
of {\bf (2)} and {\bf (3)} and Proposition \ref{prop:int_approx}  (see
Remark \ref{rk:prop1.gen}), 
\begin{align*}
     & \limsup_{n\to \infty}  \frac{P\bigl(S_0^n\in
     \lambda_n\Gamma, \,  {H_i^{(n)}} Z_i  \in   \lambda_n
     (\Gamma^\eps\setminus \Gamma) \ \text{for some} \  i   \bigr)}{P(S_0^n\in
       \lambda_n\Gamma)}  
  \leq  \frac{ \int_{-\infty}^\infty \nu\bigl( \overline{B_y^{-1} (\Gamma^\eps\setminus\Gamma)}\bigr) \ dy}
         { \int_{-\infty}^\infty \nu( B_y^{-1} \Gamma) \ dy}, 
\end{align*}
and by Assumption \ref{ass:tail_measure} and monotone convergence
theorem  the intergal in the numerator vanishes as $\eps\to 0$.

 {\bf (4) $\sim$ (5)}: \ This statement follows immediately from the
already proved equivalence {\bf (2)}$\sim${\bf (4)}, the fact that
$\Gamma \subset B_{\delta_0}(0)^c$ and Lemma \ref{l:elem.fact}.
\end{proof}

\section{Convergence of point processes}\label{sec:pt_process}

In this section we prove Theorem \ref{thm:ptprocess_longmem}. The
analysis for $N_n^1$ and for $N_n^2$ is similar, hence we will prove
the theorem at the same time for the two cases, using the 
unifying notation $G_i^{(n),j}, C_y^j, \, j=1,2$ in
\eqref{eq:unify_coeff}. We will omit
the superscript $j$ when a statement holds in both cases, and only
use the superscript $j$ when necessary. 
It is 
sufficient to prove that for any continuous function $f:\,
\bbr\times \bigl([-\infty,\infty]^d\setminus \{0\})\to [0,\infty)$ with
compact support we have
\begin{align} \label{e:laplace.f}
&E\left[\exp\left\{ -\sum_{i=-\infty}^\infty  f\big(i/n, \, G_i^{(n)}
Z_i/\lambda_n\big)\right\} \bigg| E_0(n,\Gamma)\right] \\
  \notag \xrightarrow{n\to\infty}&  
\, \frac{\int_\bbr \left( \int_{B_y^{-1}\Gamma} \exp\bigl\{ -f(y,C_yx)\bigr\}\,
              \nu(dx)\right)dy}{\int_\bbr \nu( B_y^{-1}\Gamma) dy}.              
\end{align}
Indeed, the left-hand side of \eqref{e:laplace.f} is the Laplace
transform of  the point process $N_n$,  the right-hand side of
\eqref{e:laplace.f} is the Laplace transform of the single atom point
process $\epsilon_{(I,X)}$, and convergence of the Laplace transforms
is equivalent to weak convergence of the point processes; see
Proposition 3.19 in \cite{resnick:1987}. Furthermore, using the
equivalence of (2) and (3) in Proposition \ref{prop:asym_prob_equi}
and \eqref{eq:limit_sum_tot}, we see that proving \eqref{e:laplace.f}
is the same as proving 
\begin{align} \label{e:laplace.falt}
&\frac{E\left[\exp\left\{ -\sum_{i=-\infty}^\infty  f\big(i/n, \, G_i^{(n)}
    Z_i/\lambda_n\big)\right\} \one_{E_0(n,\Gamma)}\right]}
{n P(h^{(n)}\|Z\|>\lambda_n)}\\
  \notag \xrightarrow{n\to\infty}& 
\, \frac{1}{\nu(\{z: \|z\|> 1\})}
\int_\bbr \left( \int_{B_y^{-1}\Gamma} \exp\bigl\{ -f(y,C_yx)\bigr\}\,
              \nu(dx)\right)dy.
\end{align}

The main challenge is that the expression in the left-hand side of
\eqref{e:laplace.falt} is difficult to analyze due to the complex
structure of the event $E_0(n, \Gamma)$. We solve this problem by
replacing the event $E_0(n, \Gamma)$ by simpler asymptotically
equivalent events that isolate the dominant shock. This 
is formalized in the following lemma. For positive $M,\delta$, $i\in
I_n^M$ and $\rho \in (0, 1)$ let 
\begin{equation}\label{def:pt_decom}
    \begin{aligned}
    & E_0^{(i,M)}(n, \Gamma) = \bigl\{S_0^n\in\lambda_n\Gamma, \,
{H_i^{(n)}}  Z_i \in \lambda_n \Gamma, \, \|{H_{l}^{(n)}}Z_l\|\leq \rho \lambda_n\delta \ \text{for 
  all} \ i\neq l \in I_n^M \bigr\}, \\
  &E_0^{(M)}(n, \Gamma) = \bigcup_{i\in I_n^M} E_0^{(i,M)}(n, \Gamma),\\
  & F_0^{(i,M)}(n, \Gamma) = E_0^{(i,M)}(n, \Gamma)\cap
  \bigl\{\|G_l^{(n)}Z_l\| \leq \lambda_n  \delta \ \text{for all} \ i \neq l \in I_n^M\bigr\}.    \end{aligned}
\end{equation}

\begin{lemma}\label{l:pt_process_small}
For any  nonnegative, measurable function $f: \RR\times \RR^d \to \RR$
we have 
\begin{equation}\label{eq:pt_decomp1}
\frac{E\left[\exp\left\{ -\sum_{i=-\infty}^\infty  f\big(i/n, G_i^{(n)}
  Z_i/\lambda_n\big)\right\} \left(\one_{E_0^{(M)}(n,\Gamma)} - (\sum_{|i|\leq Mn}\one_{F_0^{(i,M)}(n,\Gamma)})\right)\right]}
{n P(h^{(n)}\|Z\|>\lambda_n)} \to 0
\end{equation}

and 
\begin{equation}\label{eq:pt_decomp2}
    \begin{aligned}
         \frac{E\left[\sum_{|i|\leq Mn}\exp\left\{ -   f\big(i/n, \, G_i^{(n)}
  Z_i/\lambda_n\big)\right\} (\one_{F_0^{(i, M)}(n, \Gamma)} -\one_{\{H_i^{(n)}Z_i \in \lambda_n \Gamma\}})\right]}{nP(h^{(n)}\|Z\| > \lambda_n)} \to 0
    \end{aligned}
  \end{equation}
  as $n\to\infty$. 
\end{lemma}

\begin{proof}
 Since $f$ is nonnegative, 
 \begin{align*}
   &  \frac{E\left[\exp\left\{ -\sum_{i=-\infty}^\infty  f\big(i/n, \, G_i^{(n)}
  Z_i/\lambda_n\big)\right\} \left(\one_{E_0^{(M)}(n,\Gamma)} - (\sum_{|i|\leq Mn}\one_{F_0^{(i,M)}(n,\Gamma)})\right)\right]}
{n P(h^{(n)}\|Z\|>\lambda_n)} \\
   \leq & \frac{\sum_{|i| \leq Mn}P(H_i^{(n)} Z_i \in \lambda_n
          \Gamma) P\bigl(\|H_l^{(n)} -
          G_l^{(n)}\| \|Z_l\|> \lambda (1 - \rho)
          \delta \ \text{for some}  \ l\in I_n^M, \, l\not= i\bigr)}{nP(h^{(n)}\|Z\| > \lambda_n)}. 
 \end{align*}
 Note that 


\begin{equation*}
    \limsup_{n\to \infty} \sup_{l\in\bbz }\left\|\bigl( H_l^{(n)} -
      G_l^{(n)}/h^{(n)}\bigr)\right\| \leq  \left(\max(\|B_+\|, \|B_-\|) + 1\right)\coloneqq T,
\end{equation*}
say.  Hence 
\begin{align*}
     & \max_{i\in \bbz}  P\bigl(\|H_l^{(n)} -
          G_l^{(n)}\| \|Z_l\|> \lambda (1 - \rho)
          \delta \ \text{for some}  \ l\in I_n^M, \, l\not= i\bigr)\\
    = & O(nM P((T+1)  h^{(n)} \|Z\| > \lambda_n (1 - \rho) \delta)) = o(1)
\end{align*}
because  $h^{(n)}$ is regularly varying index $1 + \beta$  and
$\lambda_n$ satisfies \eqref{e:lambda.n}. Now \eqref{eq:pt_decomp1}
follows directly from \eqref{eq:limit_sum_tot}.  




For the second statement we write 
\begin{align*}
   & \left|{E\left[\sum_{|i|\leq Mn}\exp\left\{ -   f\big(i/n, \, G_i^{(n)}
  Z_i/\lambda_n\big)\right\} (\one_{F_0^{(i, M)}(n, \Gamma)} -\one_{\{H_i^{(n)}Z_i \in \lambda_n \Gamma\}})\right]}\right|\\
  \leq & { \sum_{|i|\leq Mn}  P(H_i^{(n)}Z_i \in \lambda_n \Gamma, S_0^n \not \in \lambda_n \Gamma)} \\
  & + { \sum_{|i|\leq Mn} P\bigl(H_i^{(n)}Z_i \in \lambda_n \Gamma, \, 
    \|H_l^{(n)}Z_l\| > \lambda_n \rho
    \delta\ \text{for some} \ l \in I_n^M, l\not= i\bigr)} \\
  & +{ \sum_{|i|\leq Mn}  P\bigl(H_i^{(n)}Z_i \in \lambda_n \Gamma,  \, \|G_l^{(n)}Z_l\| > \lambda_n \delta\ \text{for some} \ l \in I_n^M, l\not= i\bigr)}\\
  \coloneqq &  V_1 + V_2 + V_3. 
\end{align*}
We already know that 
\begin{align*}
    \lim_{n\to \infty} \frac{V_2}{nP(h^{(n)}\|Z\| > \lambda_n)} =
  \lim_{n\to \infty} \frac{V_3}{nP(h^{(n)}\|Z\| > \lambda_n)} = 0. 
\end{align*}

Finally, by the inclusion-exclusion principle, 
\begin{align*}
  & \frac{V_1}{nP(h^{(n)}\|Z\| > \lambda_n)} 
    \leq   \frac{ P\bigl(S_0^n \not \in \lambda_n\Gamma,  \, H_i^{(n)}Z_i\in \lambda_n \Gamma\ \text{for some} \ i\in I_n^M\bigr)}{nP(h^{(n)}\|Z\| > \lambda_n)} \\
    & + \frac{P\bigl( H_i^{(n)}Z_i\in \lambda_n \Gamma,  H_j^{(n)} Z_j \in
      \lambda_n \Gamma\ \text{for some} \ i,j\in I_n^M,
      i\not=j\bigr)}{nP(h^{(n)}\|Z\| > \lambda_n)}   \to 0
\end{align*}
by the equivalence between $({\bf 3})$ and $({\bf 4})$ in Proposition
\ref{prop:asym_prob_equi}, Lemma \ref{l:elem.fact} and
\eqref{eq:asym_equivalence}.  
\end{proof}

\begin{proposition} \label{pr:pp.step}
Under the assumptions of Theorem \ref{thm:ptprocess_longmem}, the
statement \eqref{e:laplace.falt} holds for any simple function in
Proposition \ref{prop:int_approx}, such that, for
each $(l,j)$ such that $b_{lj}>0$, the set $A_{lj} $ is bounded away
from the origin. 
\end{proposition}
\begin{proof}
If $\delta>0$ is small enough and $M$ is large enough, then 
\begin{align*}
&\frac{E\left[\exp\left\{ -\sum_{i=-\infty}^\infty  f\big(i/n, \, G_i^{(n)}
  Z_i/\lambda_n\big)\right\} \one_{E_0^{(M)}(n,\Gamma)}\right]}
{n P(h^{(n)}\|Z\|>\lambda_n)}\\
= & \frac{E\left[  
\exp\left\{ -\sum_{j=-\infty}^\infty  f\big(j/n, \, G_j^{(n)}
  Z_j/\lambda_n\big)\right\} (\sum_{|i|\leq Mn}\one_{F_0^{(i,M)}(n,\Gamma)})\right]}
    {n P(h^{(n)}\|Z\|>\lambda_n)}  + o(1)\\
=&\frac{E\left[\sum_{|i|\leq Mn}  
\exp\left( -   f\big(i/n, \, G_i^{(n)}
  Z_i/\lambda_n\big)\right) \one_{F_0^{(i,M)}(n,\Gamma)}\right]}
    {n P(h^{(n)}\|Z\|>\lambda_n)} + o(1) \\
 =&\frac{ E \left[\sum_{|i|\leq Mn}\exp\left(-
      f\bigl(i/n, {G_i^{(n)}} Z_i/\lambda_n\bigr)\right)
    1_{\{{H_i^{(n)}}  Z_i\in \lambda_n\Gamma\}}\right] } {n
    P(h^{(n)}\|Z\|>\lambda_n)} +o(1), 
\end{align*}
with the steps 2 and 4 following from \eqref{eq:pt_decomp1} and
\eqref{eq:pt_decomp2}. The claim of the proposition now follows,
because by Lemma 
\ref{lemma:intapprox_negligible} and the equivalence between ({\bf 3})
and ({\bf 5}) in Proposition \ref{prop:asym_prob_equi}
\begin{align*}
&\lim_{n\to\infty} \frac{E\left[\exp\left\{ -\sum_{i=-\infty}^\infty
  f\big(i/n, \, G_i^{(n)}
  Z_i/\lambda_n\big)\right\} \one_{E_0(n,\Gamma)}\right]}
{n P(h^{(n)}\|Z\|>\lambda_n)}\\
=&\lim_{M\to\infty} \lim_{n\to\infty} \frac{E\left[\exp\left\{ -\sum_{i=-\infty}^\infty
  f\big(i/n, \, G_i^{(n)}
  Z_i/\lambda_n\big)\right\} \one_{E_0^{(M)}(n,\Gamma)}\right]}
{n P(h^{(n)}\|Z\|>\lambda_n)}\\
=&\lim_{M\to\infty} \lim_{n\to\infty} \frac{ E \left[\sum_{|i|\leq Mn}\exp\left(-
      f\bigl(i/n, {G_i^{(n)}} Z_i/\lambda_n\bigr)\right)
    1_{\{{H_i^{(n)}}  Z_i\in \lambda_n\Gamma\}}\right] } {n
    P(h^{(n)}\|Z\|>\lambda_n)} \\
=&\lim_{n\to\infty} \frac{ E \left[\sum_{i=-\infty}^\infty\exp\left(-
      f\bigl(i/n, {G_i^{(n)}} Z_i/\lambda_n\bigr)\right)
    1_{\{{H_i^{(n)}}  Z_i\in \lambda_n\Gamma\}}\right] } {n
    P(h^{(n)}\|Z\|>\lambda_n)} \\
=&\frac{1}{\nu(\{z: \|z\|> 1\})}
\int_\bbr \left( \int_{B_y^{-1}\Gamma} \exp\bigl\{ -f(y,C_yx)\bigr\}\,
              \nu(dx)\right)dy, 
\end{align*}
where on the last step we used Proposition \ref{prop:int_approx}. 
\end{proof}

Suppose now that $f:\, \R\times \R^d\to [0,\infty)$ of
the form 
\begin{equation} \label{e:semisimple.f}
f(y,x)= \sum_{l=1}^L \one_{[a_{l-1},a_l)}(y) f_l(x), \ y\in\R, \, x\in\R^d,
\end{equation} 
where $-\infty<a_0<a_1<\cdots <a_L<\infty, \, L=1,2,\ldots$, and for
each $l$, $f_l:\,
 \bigl([-\infty,\infty]^d\setminus \{0\}\to [0,\infty)$ is a bounded
 continuous function with compact support. It follows from catastrophe
 theorem and the fact that any $\sigma$-finite measure can assign
 positive mass 
 to at most countably many disjoint measurable sets, that for each $l$
 we can choose 
 $\vep_l>0$ arbitrarily small such that the sets $A_{lj}$ and 
 $(pB_++(1-p)B_-)^{-1}A_{lj}$  with 
 $$  A_{lj}= \{x:\,
 (j-1)\vep_l<f_l(x)\leq j\vep_l\}
 $$
 are $\nu$-continuity sets for each
 $j=2,\ldots$, for $p=0,1$ and a.e. $0<p<1$. The same is true for
 $j=1$ if we define
 $$
 A_{l1}= \{x:\, \|x\|>\rho_l, \, 
f_l(x)\leq \vep_l\}
$$
for an appropriately chosen small enough $\rho_l>0$. Define functions
 $f_l^+$ and $f_l^-$ by $f_l^+(x) = f_l^-(x)=0$ if $f_l(x)=0$ and 
$ f_l^+(x) = j\vep_l, \, f_l^-(x) =
 (j-1)\vep_l,$ for $x\in A_{lj}$ with $j=1,\ldots,
 \|f_l\|_\infty/\eps_l+1$. The functions
 $$
 f^\pm (y,x)= \sum_{l=1}^L \one_{[a_{l-1},a_l)}(y) f^\pm_l(x), \ y\in\R, \, x\in\R^d
 $$
 satisfy the assumptions of Proposition \ref{pr:pp.step}, and
 $f^{-}(y,x)\leq f(y,x)\leq f^{+}(y,x)$ for all $y,x$. By Proposition
 \ref{pr:pp.step},
 \begin{align*}
&\frac{1}{\nu(\{z: \|z\|> 1\})}
\int_\bbr \left( \int_{B_y^{-1}\Gamma} \exp\bigl\{ -f^+(y,C_yx)\bigr\}\,
   \nu(dx)\right)dy \\
=& \lim_{n\to\infty}
   \frac{E\left[\exp\left\{ -\sum_{i=-\infty}^\infty  f^+\big(i/n, \, G_i^{(n)}
                                    Z_i/\lambda_n\big)\right\} \one_{E_0(n,\Gamma)}\right]}
{n P(h^{(n)}\|Z\|>\lambda_n)} \\
\leq & \liminf_{n\to\infty}
   \frac{E\left[\exp\left\{ -\sum_{i=-\infty}^\infty  f\big(i/n, \, G_i^{(n)}
                                    Z_i/\lambda_n\big)\right\} \one_{E_0(n,\Gamma)}\right]}
{n P(h^{(n)}\|Z\|>\lambda_n)}  \\
\leq & \limsup_{n\to\infty}
   \frac{E\left[\exp\left\{ -\sum_{i=-\infty}^\infty  f\big(i/n, \, G_i^{(n)}
                                    Z_i/\lambda_n\big)\right\} \one_{E_0(n,\Gamma)}\right]}
{n P(h^{(n)}\|Z\|>\lambda_n)}  \\
\leq & \lim_{n\to\infty}
   \frac{E\left[\exp\left\{ -\sum_{i=-\infty}^\infty  f^-\big(i/n, \, G_i^{(n)}
                                    Z_i/\lambda_n\big)\right\} \one_{E_0(n,\Gamma)}\right]}
{n P(h^{(n)}\|Z\|>\lambda_n)} \\
=&\frac{1}{\nu(\{z: \|z\|> 1\})}
\int_\bbr \left( \int_{B_y^{-1}\Gamma} \exp\bigl\{ -f^-(y,C_yx)\bigr\}\,
   \nu(dx)\right)dy.
 \end{align*}
 Letting $\vep_l\to 0$ for each $l=1,\ldots, L$, we have $f^\pm
 (y,x)\to f(y,x)$ for each $(y,x)$, and by the dominated convergence
 theorem we conclude that \eqref{e:laplace.falt} holds for any
 function of the type \eqref{e:semisimple.f}.


 Finally, let $f:\,
\bbr\times \bigl([-\infty,\infty]^d\setminus \{0\}\to [0,\infty)$ be a
continuous function with compact support. For $k=1,2,\ldots$ the
function $f^{(k)}(y,x) = f(i/k,x)$ if $i/k\leq y<(i+1)/k$, $i\in\bbz$ is a
function of the type \eqref{e:semisimple.f}. Furthermore, by the
uniform continuity of $f$ we see that $\theta_k:=\|f-f^{(k)}\|_\infty \to
0$ as $k\to\infty$. Since \eqref{e:laplace.falt} holds for  the
function $f^{(k)}$ for every $k$, we obtain, by the dominated convergence
theorem,
\begin{align*}
&\lim_{k\to\infty} \lim_{n\to\infty} \frac{E\left[\exp\left\{ -\sum_{i=-\infty}^\infty  f^{(k)}\big(i/n, \, G_i^{(n)}
                                    Z_i/\lambda_n\big)\right\} \one_{E_0(n,\Gamma)}\right]}
                              {n P(h^{(n)}\|Z\|>\lambda_n)}\\
=& \frac{1}{\nu(\{z: \|z\|> 1\})}
\int_\bbr \left( \int_{B_y^{-1}\Gamma} \exp\bigl\{ -f(y,C_yx)\bigr\}\,
              \nu(dx)\right)dy.
\end{align*}

Since $f$ is compactly supported, we can find some $y$ such that $\delta_*=\inf\{ \| x\|:\, f(y,x)>0\} > 0$.  
The fact that     \eqref{e:laplace.falt} holds for  the
function $f$ now follows from the estimate
\begin{align*}
&\left| \frac{E\left[\exp\left\{ -\sum_{i=-\infty}^\infty  f\big(i/n, \, G_i^{(n)}
                                    Z_i/\lambda_n\big)\right\} \one_{E_0(n,\Gamma)}\right]}
{n P(h^{(n)}\|Z\|>\lambda_n)} \right. \\
&\left. -\frac{E\left[\exp\left\{ -\sum_{i=-\infty}^\infty  f^{(k)}\big(i/n, \, G_i^{(n)}
                                    Z_i/\lambda_n\big)\right\} \one_{E_0(n,\Gamma)}\right]}
{n P(h^{(n)}\|Z\|>\lambda_n)}\right|\\
\leq& \theta_k \frac{\sum_{i=-\infty}^\infty P\bigl( \| H_i^{(n)}
      Z_i\|>\lambda_n \delta_*\bigr)}{n P(h^{(n)}\|Z\|>\lambda_n)}
\end{align*}
and \eqref{eq:limit_sum_tot} (with $\Gamma$ replaced by
$B_{\delta_*}^c$).


\section{Cluster Size}\label{sec:cluster_size}

In this section we prove Theorem \ref{thm:cluster_size}. The statement
\eqref{eq:clusterSize} describes the probability that all partial sums
$\sum_{k=j}^{j+n-1} X_k$ with $j$ in the range
$[-n\theta_2,n\theta_1]$ are unusually large. The catastrophe
principle suggests that such a rare event is most likely due to a
single exceptionally large value of a noise variable. Considering
first only the noise variables $Z_i$ with $|i|\leq nM$, we introduce
the event 
\begin{equation} \label{e:i.range}
\cE_n^M(\Psi) =
\left\{
\text{there is} \ i \in I_n^M \ \text{such that} \ 
H_{i-j}^{(n)} Z_i \in \lambda_n \Psi \ 
\text{for all} \  j \in [-n\theta_2,n\theta_1]
\right\}.
\end{equation}
We start by connecting this event to the point process $N_n^2$ in
\eqref{e:point.pr}. 

\begin{lemma}\label{lemma:cluster_equi_pt}
Suppose that $\Gamma$  satisfies Assumption \ref{ass:tail_measure}. Let
$0<\delta<\delta_0$ and denote 
\[
  D^{M,\pm\delta}(\theta_1,\theta_2) = \bigl\{(y, z): y \in [-M, M],\, B_{y-t}z
  \in \Psi^{\pm \delta} \ \text{for all} \
  t \in [-\theta_2, \theta_1]\cap \bbq\bigr\}.
\] 
Then for any $M^\prime >M$, $   0<\theta_1^\prime<\theta_1, \,
0<\theta_2^\prime<\theta_2$ we have 
\begin{equation} \label{e:connect.N2}
  \limsup_{n\to \infty} P\bigl(\cE_n^M(\Psi) \big|E_0(n,\Gamma)\bigr)
    - P\bigl(N_n^2\bigl(D^{M^\prime,\delta}(\theta_1^\prime, \theta_2^\prime\bigr)
    \geq 1\big|E_0(n,\Gamma)\bigr) \leq 0 
  \end{equation}
and for any $0<M^\prime<M$, $   \theta_1^\prime>\theta_1, \,
\theta_2^\prime>\theta_2$ we have
\begin{equation} \label{e:connect.N2a}
  \liminf_{n\to \infty} P\bigl(\cE_n^M(\Psi) \big|E_0(n,\Gamma)\bigr)
    - P\bigl(N_n^2\bigl(D^{M^\prime,-\delta}(\theta_1^\prime, \theta_2^\prime\bigr)
    \geq 1\big|E_0(n,\Gamma)\bigr) \geq 0.  
  \end{equation}
\end{lemma}

\begin{proof}
Fix   $M^\prime >M$, $   0<\theta_1^\prime<\theta_1, \,
0<\theta_2^\prime<\theta_2$ and let 
$M^{\prime\prime}\in (M,M^\prime)$. 
Let $\eps>0$ be a small number. For $-M/\eps\leq k\leq M/\eps$ let $J_k^{n, \eps} =
        [kn\eps, (k+1) n\eps) \cap \ZZ$ and denote $I_n^{M,\eps} =
        \bigcup_{-M/\eps\leq k\leq M/\eps} J_k^{n, \eps} $.  It is
        easy to check that
$$
[-(M-\eps)n,(M-\eps)n]\subseteq        I_n^{M,\eps} \subseteq
[-(M+\eps)n,(M+\eps)n].
$$
Let $\cE_n^{M,\eps}(\Psi)$ be defined identically to $\cE_n^M(\Psi)$
but with  $I_n^{M,\eps}$ replacing $I_n^M$. Let $\delta>0$ be a small
number. We conclude that for $0<\eps<\min(M^{\prime\prime} -M,
\theta_1-\theta_1^\prime, \theta_2-\theta_2^\prime)$ 
and all $n$ large
enough we have
\begin{align*}
 &P\bigl(\cE_n^M(\Psi) \big|E_0(n,\Gamma)\bigr) \leq 
  P\bigl(\cE_n^{M^\prime,\eps}(\Psi) \big|E_0(n,\Gamma)\bigr) \\
  \leq &P\bigl( \text{there is} \ -M^{\prime\prime}/\eps \leq k\leq
                    M^{\prime\prime}/\eps \ \text{and} \ i\in J_k^{n,
   \eps} \ \text{with}  \\
                  & \hskip 0.2in H^{(n)}_{[kn\eps]-j} Z_i\in \lambda_n \Psi^{\delta/4} \
                   \text{for all} \  j\in [-n\theta_2,n\theta_1] \big|E_0(n,\Gamma)\bigr)\\
\leq &P\bigl( \text{there is} \ -M^{\prime\prime}/\eps \leq k\leq
                    M^{\prime\prime}/\eps \ \text{and} \ i\in J_k^{n,
   \eps} \ \text{with}  \ H^{(n)}_{[(k-l)n\eps]} Z_i\in \lambda_n \Psi^{\delta/4} \\
                  & \hskip 0.2in 
                   \text{for all} \   -\theta_2^\prime/\eps\leq l\leq \theta_1^\prime/\eps \big|E_0(n,\Gamma)\bigr)
                    + P\bigl( R_1^{n, \eps}\big|E_0(n,\Gamma)\bigr)\\
\leq &P\bigl( \text{there is} \ -M^{\prime\prime}/\eps \leq k\leq
                    M^{\prime\prime}/\eps \ \text{and} \ i\in J_k^{n,
   \eps} \ \text{with}  \ B_{(k-l)\eps}h_m Z_i\in \lambda_n \Psi^{\delta/2} \\
                  & \hskip 0.2in  
                   \text{for all} \   -\theta_2^\prime/\eps\leq l\leq \theta_1^\prime/\eps \big|E_0(n,\Gamma)\bigr)
                    + P\bigl( R_1^{n,
                    \eps}\big|E_0(n,\Gamma)\bigr)+P\bigl( R_2^{n,
                    \eps}\big|E_0(n,\Gamma)\bigr) \\
\leq &P\Bigl( N_n^2\bigl(D_{\eps,n}^{M^{\prime\prime},\delta}(\theta_1^\prime,\theta_2^\prime)\bigr)\geq 1\big| E_0(n,\Gamma)\Bigr)
 + P\bigl( R_1^{n,
                    \eps}\big|E_0(n,\Gamma)\bigr)+P\bigl( R_2^{n,
       \eps}\big|E_0(n,\Gamma)\bigr) \\
  \leq &P\Bigl( N_n^2\bigl(D^{M^{\prime},\delta}(\theta_1^\prime,\theta_2^\prime)\bigr)\geq 1\big| E_0(n,\Gamma)\Bigr)
 + \sum_{l=1}^3 P\bigl( R_l^{n,
                    \eps}\big|E_0(n,\Gamma)\bigr), 
\end{align*}
where 
\begin{align*}
D_{\eps,n}^{M^\prime,\delta}(\theta_1^\prime,\theta_2^\prime) = &\bigl\{ (y,z):\,\text{there is} \ -M^\prime/\eps \leq k\leq
                   M^\prime/\eps \ \text{and} \; y\in J_k^{n,   \eps} /n \\
                   \
 & \text{with} \ B_{y-t}^\eps z\in \Psi^{\delta/2} \ \text{for all} \
  -\theta_2^\prime\leq t\leq \theta_1^\prime\bigr\}, 
\end{align*}
and the ``remainder'' events are defined by 
\begin{align*}
  R_1^{n, \eps}=& \bigl\{ \text{there is} \ -M^\prime/\eps \leq k\leq
                   M^\prime/\eps \ \text{and} \ i\in J_k^{n,
   \eps}  \\
   &\text{with} \ \| H_{[kn\eps]-j}-H_{i-j}\|\| Z_i\|
                          >\lambda_n\delta/4 \ \text{for some} \ j\in
                           [-n\theta_2,n\theta_1]\bigr\}, \\
 R_2^{n, \eps}=& \bigl\{ \text{there is} \   -M^\prime/\eps \leq k\leq
                   M^\prime/\eps \ \text{and} \ i\in J_k^{n,
   \eps} \\
&\text{with}  \ \|H^{(n)}_{[(k+l)n\eps]} -B_{(k+l)\eps}
                  h^{(n)}\|\|Z_i\|>\delta/4 \ 
                     \text{for some} \ -\theta_1^\prime/\eps\leq
                    l\leq\theta_2^\prime/\eps \bigr\}, \\
R_3^{n, \eps}=& \bigl\{ \text{there is} \ y\in [-M^{\prime},M^{\prime}] \
  \text{with} \ \| B_{y-t}-B_{y-t}^\eps\|h_n \|
                Z_{[yn]}\|>\lambda_n\delta/2 \\
   &\text{for some} \
                -\theta_2^\prime\leq t\leq \theta_1^\prime \bigr\}. 
   \end{align*}
We see that the claim \eqref{e:connect.N2} will follow once we prove
that 
   \begin{equation} \label{e:remainders.small}
 \lim_{\eps\to 0} \limsup_{n\to\infty} P\bigl( R_d^{n,
                    \eps}\big|E_0(n,\Gamma)\bigr) =0 \ \text{for} \ d=1,2,3.
                \end{equation}

For $d=1$ we use \eqref{eq:coeff_group_err} to note that for a fixed
$\eps>0$ and all $n$ large enough, for all relevant $k$ in $j$ we have 
  $\| H_{[kn\eps]-j}-H_{i-j}\| \leq f(\eps) h_n $ for  some $f(\eps)
  \to 0$ as $\eps \to 0$.  Therefore, by \eqref{eq:asym_equivalence}, 
 \begin{align*}
\limsup_{n\to\infty} P\bigl( R_1^{n,
                    \eps}\big|E_0(n,\Gamma)\bigr) 
\leq& 2M^\prime (\theta_1+\theta_2)\limsup_{n\to\infty}  \frac{nP\bigl(f(\eps) h^{(n)} \|Z\|
    > \lambda_n \delta/4\bigr)}{P\bigl( E_0(n,\Gamma)\bigr)} \\
=&O((f(\eps))^\alpha),
  \end{align*}
and   \eqref{e:remainders.small} with $d=1$ follows. 

Next, the continuity of $B_y$ as a function of $y$ means that
$$
\sup_{y\in [-M^\prime,M^\prime], \, t \in [-\theta_2^\prime,
  \theta_1^\prime]} \|B_{y-t} - B_{y - t}^\eps\| \leq g(\eps)
$$
 for some $g(\eps)\to 0$ as $\eps\to 0$. Therefore, repeating the
 argument in the case $d=1$ we see that for a fixed $\eps>0$,
 $$
 \limsup_{n\to\infty} P\bigl( R_3^{n,
                    \eps}\big|E_0(n,\Gamma)\bigr) = O((g(\eps))^\alpha),
                  $$
giving us \eqref{e:remainders.small} with $d=3$, and so it remains to
prove \eqref{e:remainders.small} for $d=2$. Consider all relevant
choices of $k$ and $l$. Taking the case $k+l\leq 0$ first, we write 
\begin{align*}
     & \sup_{k,l:\, k+l\leq 0}\left\|H_{[(k+l)n\eps]}^{(n)}/h^{(n)}-
       B_{(k+l)\eps}\right\| \\
&\leq  \sup_{k,l:\, k+l\leq 0} \left\|H_0^{n-[(k+l)n\eps]} /h^{(n)}- (1-(k+l)\eps)^{1+\beta}B_+ \right\|  \\
     & +  \sup_{k,l:\, k+l\leq 0} \left\|H_0^{[(-(k+l) n\eps)]}/h^{(n)}
     - (-(k+l)\eps)^{1+\beta} B_+ \right\| =   o(1)
\end{align*}
as we have seen before (in the proof of Lemma
\ref{lemma:asym_lim_meas}). Since the case $k+l>0$ is similar, 
 \eqref{e:remainders.small} for $d=2$ follows using the same arguments
 as for the case $d=1$. 

 This proves \eqref{e:connect.N2}, and \eqref{e:connect.N2a} can be
 proved in the same way. 
\end{proof}

We are now ready to prove Theorem \ref{thm:cluster_size}.  Let
$\delta>0$ be a small number. Write
\begin{align} \label{e:split.upper.cl}
    & P\big(J_n^{\Psi,+} > n\theta_1,\, J_n^{\Psi,-} < -n\theta_2 \,\big|\, E_0(n,\Gamma)\big) \\
       \leq  &  P\bigl(  \cE_n^M(\Psi^\delta)\big|E_0(n,\Gamma)\bigr)
               +  P\bigg((\cE_n^M(\Psi^\delta))^c\cap \bigcap_{j \in [-n\theta_2,n\theta_1]}
               E_j(n, \Psi) \Big|E_0(n,\Gamma)\biggr). \notag 
\end{align}
For any 
$M^\prime >M$, $   0<\theta_1^\prime<\theta_1, \,
0<\theta_2^\prime<\theta_2$, by Lemma \ref{lemma:cluster_equi_pt} and
Theorem \ref{thm:ptprocess_longmem} we have 
\begin{align} \label{e:lims.E}
\limsup_{n\to\infty} \, &P\bigl(  \cE_n^M(\Psi^\delta)\big|E_0(n,\Gamma)\bigr)
\leq \limsup_{n\to\infty} P\bigl(N_n^2\bigl(D^{M^\prime,2\delta}(\theta_1^\prime, \theta_2^\prime)\bigr)
    \geq 1\big|E_0(n,\Gamma)\bigr) \\
\notag \leq &P\Bigl( (I,X^2)\in \overline{D^{M^\prime,2\delta}(\theta_1^\prime, \theta_2^\prime)}\Bigr)
\leq  \frac{\int_{-\infty}^\infty  \;
        \nu\!\Big(B_y^{-1} \Gamma \cap \bigcap_{-\theta_2^\prime\le t \le \theta_1^\prime} \!\!  B_{y-t}^{-1}\Psi^{3\delta}\Big) dy}
        {\int_{-\infty}^\infty \nu(B_y^{-1} \Gamma) \,dy} .
\end{align}
Next, we write 
\begin{align*}
P\bigg((\cE_n^M(\Psi^\delta))^c\cap \bigcap_{j \in [-n\theta_2,n\theta_1]}
               E_j(n, \Psi) \Big|E_0(n,\Gamma)\biggr)
= P(B_{n,M})/P\bigl(E_0(n,\Gamma)\bigr),
\end{align*}
where
\begin{align*}
B_{n,M} =&\bigg\{ \text{for every} \ i\in I_n^M \ \text{there is} \
  j(i)\in [-n\theta_2,n\theta_1] \ \text{with} \ H_{i-j(i)}^{(n)}Z_i\notin
  \lambda_n \Psi^\delta \\
&\hskip 0.2 in\text{but} \ \sum_{m=-\infty}^\infty
  H_{m-j}^{(n)}Z_m \in \lambda_n\Psi \ \text{for all} \ j \in
  [-n\theta_2,n\theta_1]\bigg\}. 
\end{align*}
Let $\eps>0$ be a small number. By Lemma \ref{l:elem.fact},
Proposition \ref{prop:asym_prob_equi} and Lemma \ref{lemma:intapprox_negligible}
\begin{align} \label{e:Bnm.one}
P(B_{n,M}) = P\bigl( B_{n,M} \cap \{ \| H_i^{(n)} Z_i\|>\lambda_n\eps
  \ \text{for exactly one} \ i\in\bbz\}\bigr) +P(\hat B_{n,M}), 
\end{align}
where $\hat B_{n,M}$ denotes the complement of the first event within $B_{n,M}$, and satisfies 
\begin{equation} \label{e:Bnm1.1}
\lim_{M\to\infty} \limsup_{n\to\infty} P(\hat
B_{n,M})/P\bigl(E_0(n,\Gamma)\bigr)=0.
\end{equation}
Furthermore,
\begin{align} \label{e: Bnm1.2}
&P\bigl( B_{n,M} \cap \{ \| H_i^{(n)} Z_i\|>\lambda_n\eps
  \ \text{for exactly one} \ i\in\bbz\}\bigr) \\
 \notag  \leq\sum_{i=-\infty}^\infty &P\biggl(  \| H_i^{(n)}
  Z_i\|>\lambda_n\eps, \ \| H_m^{(n)} Z_m\|\leq \lambda_n\eps \
  \text{for all} \ m\not=i \\
\notag & \hskip 1in \text{and} \ 
\left\| \sum_{m\not= i} H_{m-j}^{(n)} Z_m\right\| >\lambda_n\delta \ \text{for some} \ j \in
  [-n\theta_2,n\theta_1]\biggr)\\
\notag \leq \sum_{i=-\infty}^\infty &P\bigl(  \| H_i^{(n)}
  Z_i\|>\lambda_n\eps\bigr) \, \sup_i P\biggl( \| H_m^{(n)} Z_m\|\leq \lambda_n\eps \
  \text{for all} \ m\not=i \\
\notag & \hskip 1in \text{and} \ 
\left\| \sum_{m\not= i} H_{m-j}^{(n)} Z_m\right\| >\lambda_n\delta \ \text{for some} \ j \in
  [-n\theta_2,n\theta_1]\biggr). 
\end{align}

We claim that 
\begin{align} \label{e:last.upper}
&\lim_{n\to\infty}\sup_i P\biggl( \| H_m^{(n)} Z_m\|\leq \lambda_n\eps
                                     \ \text{for all} \ m\not=i \\
 \notag &\hskip 1in  \text{and} \ 
\left\| \sum_{m\not= i} H_{m-j}^{(n)} Z_m\right\| >\lambda_n\delta \ \text{for some} \ j \in
[-n\theta_2,n\theta_1]\biggr) = 0.
\end{align}

Assume, for a moment, that \eqref{e:last.upper} holds. Then  by
\eqref{e:split.upper.cl} - \eqref{e:last.upper}, for any $\delta>0$
and $   0<\theta_1^\prime<\theta_1, \,
0<\theta_2^\prime<\theta_2$, 
\begin{align*}
&\limsup_{n\to\infty} P\big(J_n^{\Psi,+} > n\theta_1,\, J_n^{\Psi,-} <
-n\theta_2 \,\big|\, E_0(n,\Gamma)\big) \\
&\hskip 1in\leq \frac{\int_{-\infty}^\infty  \;
        \nu\!\Big(B_y^{-1} \Gamma \cap \bigcap^*_{-\theta_2^\prime\le t \le \theta_1^\prime} \!\!  B_{y-t}^{-1}\Psi^{3\delta}\Big) dy}
        {\int_{-\infty}^\infty \nu(B_y^{-1} \Gamma) \,dy}.
\end{align*}
Letting $\theta_i^\prime \uparrow \theta_i, \, i=1,2$ and using
continuity of the matrix-valued function $B$, we see that for any
$\delta>0$, 
\begin{align*}
&\limsup_{n\to\infty} P\big(J_n^{\Psi,+} > n\theta_1,\, J_n^{\Psi,-} <
-n\theta_2 \,\big|\, E_0(n,\Gamma)\big) \\
&\hskip 1in \leq \frac{\int_{-\infty}^\infty  \;
        \nu\!\Big(B_y^{-1} \Gamma \cap \bigcap^*_{-\theta_2\le t \le \theta_1} \!\!  B_{y-t}^{-1}\Psi^{4\delta}\Big) dy}
        {\int_{-\infty}^\infty \nu(B_y^{-1} \Gamma) \,dy}.
\end{align*}
Letting $\delta\downarrow 0$ and using
Assumption~\ref{ass:tail_measure} gives us the upper bound  
\begin{align} \label{e:upper.th2}
&\limsup_{n\to\infty} P\big(J_n^{\Psi,+} > n\theta_1,\, J_n^{\Psi,-} <
-n\theta_2 \,\big|\, E_0(n,\Gamma)\big) \\
&\hskip 1in \leq \frac{\int_{-\infty}^\infty  \;
        \nu\!\Big(B_y^{-1} \Gamma \cap \bigcap^*_{-\theta_2\le t \le \theta_1} \!\!  B_{y-t}^{-1}\Psi\Big) dy}
        {\int_{-\infty}^\infty \nu(B_y^{-1} \Gamma) \,dy}. \notag 
      \end{align}
  We now check \eqref{e:last.upper}. Since for any $\rho  > 0$, $\sup_i P(\|H_i^{(n)} Z_i\| > \lambda_n \rho) \to 0$, it is clearly  enough to prove that
\begin{align*}      
&P\biggl( \| H_m^{(n)} Z_m\|\leq \lambda_n\eps
                                     \ \text{for all} \ m\\
 &\hskip 0.6in \text{and} \ 
\left\| \sum_{m} H_{m-j}^{(n)} Z_m\right\| >\lambda_n\delta/2, \ \text{some} \ j \in
[-n\theta_2,n\theta_1]\biggr) \to 0
\end{align*}
as $n\to\infty$, and this holds because the latter probability is bounded from above by
\begin{align*}
&P\bigl( \| S_j^n\|>\lambda_n\delta/2 \ \text{for some} \ j \in
[-n\theta_2,n\theta_1]\big) \\
\leq &P\Bigl( \| S_0^m\|>\lambda_n\delta/2 \ \text{for some} \ m \
  \text{with} \ |m|\leq n\bigl( 1+\max(\theta_1,\theta_2)\bigr)\Bigr)
  \to 0
\end{align*}
by \eqref{e:lambda.n} and the Marcinkiewicz-Zygmund strong law of
large numbers; see e.g. \cite{loeve:1977}, p. 255. 

It remains to prove a lower bound matching \eqref{e:upper.th2}. Let
$\delta>0$. Write
\begin{align} \label{e:split.lower.cl}
    & P\big(J_n^{\Psi,+} > n\theta_1,\, J_n^{\Psi,-} < -n\theta_2 \,\big|\, E_0(n,\Gamma)\big) \\
       \geq  &   P\bigg(\cE_n^M(\Psi^{-\delta})\cap \bigcap_{j \in [-n\theta_2,n\theta_1]}
               E_j(n, \Psi) \Big|E_0(n,\Gamma) \biggr) \notag \\
=& P\bigl( \cE_n^M(\Psi^{-\delta}) \big|E_0(n,\Gamma)\bigr) - P\bigg(\cE_n^M(\Psi^{-\delta})\cap \bigcup_{j \in [-n\theta_2,n\theta_1]}
               E_j(n, \Psi)^c \Big|E_0(n,\Gamma)\biggr). \notag 
\end{align}
For any 
$0<M^\prime <M$, $   \theta_1^\prime>\theta_1, \,
\theta_2^\prime>\theta_2$, by Lemma \ref{lemma:cluster_equi_pt} and
Theorem \ref{thm:ptprocess_longmem} we have 
\begin{align} \label{e:limi.E}
\liminf_{n\to\infty} \, &P\bigl(  \cE_n^M(\Psi^{-\delta})\big|E_0(n,\Gamma)\bigr)
\geq \liminf_{n\to\infty} P\bigl(N_n^2\bigl(D^{M^\prime,{-2\delta}}(\theta_1^\prime, \theta_2^\prime)\bigr)
    \geq 1\big|E_0(n,\Gamma)\bigr) \\
\notag \geq &P\Bigl( (I,X^2)\in \bigl( D^{M^\prime,-2\delta}(\theta_1^\prime, \theta_2^\prime)\bigr)^\circ\Bigr)
\geq  \frac{\int_{-M^\prime}^{M^\prime}  \;
        \nu\!\Big(B_y^{-1} \Gamma \cap \bigcap_{-\theta_2^\prime\le t \le \theta_1^\prime} \!\!  B_{y-t}^{-1}\Psi^{-3\delta}\Big) dy}
        {\int_{-\infty}^\infty \nu(B_y^{-1} \Gamma) \,dy} .
\end{align}
Furthermore,
\begin{align*}
P\bigg(\cE_n^M(\Psi^{-\delta})\cap \bigcup_{j \in [-n\theta_2,n\theta_1]}
               E_j(n, \Psi)^c \Big|E_0(n,\Gamma)\biggr) \leq P\bigl(
  \tilde B_{n,M}\bigr)/P\bigl( E_0(n,\Gamma)\bigr),
\end{align*}
where
\begin{align*}
\tilde B_{n,M} =&\bigg\{ \text{there is} \ i\in I_n^M \ \text{with} \ H_{i-j}^{(n)}Z_i\in
  \lambda_n \Psi^{-\delta } \ \text{for all} \ 
  j\in [-n\theta_2,n\theta_1]   \\
&\hskip 0.2 in\text{but} \ \left\| \sum_{m\not= i}
  H_{m-j}^{(n)}Z_m \right\| > \lambda_n \delta \ \text{for some} \ j \in
  [-n\theta_2,n\theta_1]\bigg\}. 
\end{align*}
An argument analogous to the one used to prove the upper bound
\eqref{e:upper.th2} shows that
$$
P\bigl(
  \tilde B_{n,M}\bigr)/P\bigl( E_0(n,\Gamma)\bigr)\to 0 \ \text{as} \
  n\to\infty.
  $$
  Therefore, \eqref{e:split.lower.cl} and \eqref{e:limi.E} tell us
  that for any $\delta>0$, $0<M^\prime <M$, $   \theta_1^\prime>\theta_1, \,
\theta_2^\prime>\theta_2$
  \begin{align*}&\liminf_{n\to\infty} P\big(J_n^{\Psi,+} > n\theta_1,\, J_n^{\Psi,-} <
-n\theta_2 \,\big|\, E_0(n,\Gamma)\big) \\ 
  &\hskip 1in     \geq \frac{\int_{-M^\prime}^{M^\prime}  \;
        \nu\!\Big(B_y^{-1} \Gamma \cap \bigcap_{-\theta_2^\prime\le t \le \theta_1^\prime} \!\!  B_{y-t}^{-1}\Psi^{-3\delta}\Big) dy}
        {\int_{-\infty}^\infty \nu(B_y^{-1} \Gamma) \,dy}. 
 \end{align*}
 Now the lower bound
 \begin{align*}&\liminf_{n\to\infty} P\big(J_n^{\Psi,+} > n\theta_1,\, J_n^{\Psi,-} <
-n\theta_2 \,\big|\, E_0(n,\Gamma)\big) \\ 
  &\hskip 1in     \geq \frac{\int_{-\infty}^{\infty}  \;
        \nu\!\Big(B_y^{-1} \Gamma \cap \bigcap_{-\theta_2\le t \le \theta_1} \!\!  B_{y-t}^{-1}\Psi\Big) dy}
        {\int_{-\infty}^\infty \nu(B_y^{-1} \Gamma) \,dy}. 
 \end{align*}
matching \eqref{e:upper.th2} follows by letting $M\to\infty$ (and
choosing, say, $M^\prime=M/2$), $\theta_i^\prime\to\theta_i, \, i=1,2$
and $\delta\to 0$. 

This completes the proof of Theorem \ref{thm:cluster_size}.

\section{Some estimates of the tails of sums}\label{sec:lemmas}

In this section we derive some estimates of the tails of certain sums
that are used in key places in the paper. The following lemma is an
extension of Lemma 2 in \cite{wang:samorodnitsky:2025}. 

\begin{lemma} \label{lemma:WLLN_longmem}
Let $(Z_i)$  be i.i.d.  zero mean one-dimensional random variables that are 
  regularly varying with index $\alpha>1$. Let $\{\lambda_n\}$ be a
  sequence of positive real numbers. For each $n=1,2,\ldots$ let 
  $\{\beta_{in}\}_{i =
    -\infty}^\infty $ be a sequence satisfying for all  
  $n$, $|\beta_{in}| \leq
  \lambda_n$  for all $i$, and $\sum_{i = -\infty}^\infty
  |\beta_{in}|^{\alpha - s} = o(\lambda_n^{\alpha -s})$ for all small
  enough $s > 0$. Then for any $\delta, \tau> 0$, $s>0$ small enough
  and all $n\geq
n_0(s,\tau/B)$ (that may also depend on the distribution of $Z$), we have 
   \begin{equation}\label{eq:WLLN_bdd}
  \begin{aligned}
    &  P\left(|\sum_{i = -\infty}^\infty \beta_{in} Z_i\one(|\beta_{in} Z_i| \leq \lambda_n \tau)| > \delta\lambda_n\right) \\
      \leq & \begin{cases}
    \left(\frac{(1+s)\sum_{i = -\infty}^\infty |\beta_{in}|^{\alpha -
          s}}{ \delta \tau^{\alpha-s-1} \lambda_n^{\alpha -
          s}}\right)^{\delta/4\tau} & \text{ if } \ 1< \alpha \leq 2\\
    \left(\frac{E[Z^2] \sum_{i=-\infty}^\infty\beta_{in}^2}{\delta
        \tau\lambda_n^2} \right)^{\delta/4\tau} & \text{ if } \ \alpha > 2
      \end{cases}
  \end{aligned}  
\end{equation}
and 
  \begin{equation}\label{eq:WLLN_tot}
  \begin{aligned}
       & P\left(|\sum_{i = -\infty}^\infty \beta_{in} Z_i| > \delta\lambda_n\right) \\
       \leq & \begin{cases}
           (1+s) \lambda_n^{- \alpha + s}  \sum_{i = -\infty}^\infty
           \beta_{in}^{\alpha - s}  +   \left(\frac{(1+s)\sum_{i = -\infty}^\infty |\beta_{in}|^{\alpha - s}}{ \delta \tau^{\alpha-s-1} \lambda_n^{\alpha - s}}\right)^{\delta/4\tau}
 & \text{ if } \ 1 < \alpha \leq 2\\
            (1+s) \lambda_n^{- \alpha + s}  \sum_{i = -\infty}^\infty
            \beta_{in}^{\alpha - s}  +  \left(\frac{E[Z^2]
                \sum_{i=-\infty}^\infty\beta_{in}^2}{\delta
                \tau\lambda_n^2} \right)^{\delta/4\tau} 
& \text{ if }\ \alpha > 2
       \end{cases} 
  \end{aligned}
  \end{equation}
 
\end{lemma}

The proof of the lemma relies on the following concentration
inequality due to \cite{prokhorov:1959}: let $Y_1,\ldots, Y_m$ be
independent zero mean random variables 
   such that for some $c>0$ we have 
   $Pr(|Y_i|\leq c) = 1$ for all $i$. Then for any $t> 0$, 
   $$P(Y_1+ \cdots + Y_m>t)\leq \left(\frac{ct}{\sum_{i=1}^m \Var(Y_i)}\right)^{-t/2c},$$
and the bound remains valid for $m=\infty$ as long as the series
converges. 
\begin{proof}[Proof of the lemma] 
By \eqref{e:cond.conv} the infinite sums in the statement of the lemma
converge. For \eqref{eq:WLLN_bdd}, let $\tau>0$ and denote $Z_i^* =
Z_i\one(|\beta_{in}Z_i| \leq \lambda_n \tau)$. We use the
concentration inequality with $Y_i = \beta_{in} Z_i^* - E[\beta_{in}
Z_i^*]$, $c = 2\lambda_n \tau, \, t = \delta \lambda_n /2$.  Note that
by the Karamata theorem, for some regularly varying with exponent
$1-\alpha$ function $g$, $s>0$ and all large $n$, 
$$
\big| E[Z_i^*]\big|  = g\bigl(\lambda_n\tau/|\beta_{in}|\bigr) \leq (1+s)
\bigl(\lambda_n\tau/|\beta_{in}|\bigr)^{1-\alpha + s},
$$ 
with the last step using the Potter bounds. Therefore, for small
enough $s>0$ and large enough $n$, by assumption $\sum_{i = -\infty}^\infty |\beta_{in}|^{\alpha - s} = o(\lambda_n^{\alpha -s})$,
$$
\left|\sum_{i} \beta_{in} E[Z_i^*]\right| \leq
(1+s)(\lambda_n\tau)^{1-\alpha -s} \sum_i \beta_{in}^{\alpha - s}
\leq \delta \lambda_n/2.
$$
If  $1< \alpha \leq 2$, then for some regularly varying with exponent
$2-\alpha$ function $h$, $s>0$ and all large $n$, 
\begin{align*}
    \sum_i \Var(Y_i) = \sum_{i} \beta_{in}^2h\bigl(
  \lambda_n\tau/|\beta_{in}|\bigr) \leq (1+s)(\lambda_n\tau)^{2-\alpha
  + s} \sum_i |\beta_{in}|^{\alpha - s},
\end{align*}
and the concentration inequality gives us \eqref{eq:WLLN_bdd}. If
$\alpha > 2$, then we use 
$$
\sum_i Var(Y_i) = \sum_i \beta_{in}^2 Var(Z_i^*)  \leq E[Z_0^2] \sum_i
\beta_{in}^2, 
$$
and the concentration inequality again gives us \eqref{eq:WLLN_bdd}.

Finally, \eqref{eq:WLLN_tot}  follows from \eqref{eq:WLLN_bdd} by
writing 
\begin{align*}
 P\left( \left| \sum_{i=-\infty}^\infty \beta_{in} Z_i
       \right| > n\delta\right) \leq& P\bigl( \max_i |\beta_{in}Z_i| >
  n\tau \bigr) \\
+& P\left( \left| \sum_{i=-\infty}^\infty \beta_{in} Z_i
        \one(|\beta_{in}Z_i|\leq \tau n)\right| > n\delta\right).
\end{align*}
and using the Potter bounds in 
\begin{align*}
    P(\max_i |\beta_{in} Z_i |> n\tau)\leq& \sum_{i} P(|\beta_{in} Z_i |> n\tau) 
\leq (1+s) \lambda_n^{-\alpha  + s} \sum_{i}
     |\beta_{in}|^{\alpha-s}
\end{align*}
for all large $n$.     
\end{proof}




\end{document}